\documentclass[10pt]{article}

\usepackage{amssymb,latexsym,amsmath,enumerate,verbatim,amsfonts,amsthm}
\usepackage{algorithm,algorithmic,cite,multicol,microtype}
\usepackage{multirow,diagbox}
\usepackage[figuresright]{rotating} 
\usepackage{color} 
\usepackage{graphicx}
\usepackage[active]{srcltx}

\allowdisplaybreaks[4]
\numberwithin{equation}{section}

\usepackage[framemethod=tikz]{mdframed}
\mdfsetup{%
  skipbelow=4pt,
  skipabove=8pt,
  linewidth=1.25pt,
  backgroundcolor=gray!10,
  userdefinedwidth=\textwidth,
  roundcorner=10pt,
}

\newtheorem{lemma}{Lemma}[section]
\newtheorem{definition}{Definition}[section]

\newtheorem{theorem}{Theorem}[section]
\newtheorem{corollary}{Corollary}[section]
\newtheorem{proposition}{Proposition}[section]
\newtheorem{remark}{Remark}[section]
\newtheorem{example}{Example}[section]
\newtheorem{assumption}{Assumption}[section]

\def\R{{\rm I\!R}}

\def\argmin{\mathop{\rm arg\,min}}
\def\Argmin{\mathop{\rm Arg\,min}}

\def\Diag{{\rm Diag}}
\def\tx{{\widetilde x}}

\def\xfeas{x^\odot}
\def\betamin{{\rm \beta_{min}}}
\def\sigmamin{{\rm\tilde{\sigma}_{min}}}

\def\d{{\rm dist}}
\def\dom{{\rm dom}\,}
\def\xorig{{x_{\rm orig}}}
\def\xfeasss{x^\circledcirc}

\title{\sf Retraction-based first-order feasible methods for difference-of-convex programs with smooth inequality and simple geometric constraints}

\author{
Yongle Zhang\thanks{
        Department of Mathematics, Visual Computing and Virtual Reality Key Laboratory of Sichuan Province, Sichuan Normal University, Chengdu, People's Republic of China.
        E-mail: \texttt{yongle-zhang@163.com}}
\and
Guoyin Li\thanks{
        Department of Applied Mathematics, University of New South Wales, Sydney, Australia.
        E-mail: \texttt{g.li@unsw.edu.au}}
\and
Ting Kei Pong\thanks{
		Department of Applied Mathematics, the Hong Kong Polytechnic University, Hong Kong, People's Republic of China.
		E-mail: \texttt{tk.pong@polyu.edu.hk}}
\and
Shiqi Xu\thanks{
        Department of Mathematics, Sichuan Normal University, Chengdu, People's Republic of China.
        E-mail: \texttt{xushiq98@163.com}}
}

\begin{document}

\maketitle

\begin{abstract}
  In this paper, we propose first-order feasible methods for difference-of-convex (DC) programs with smooth inequality and simple geometric constraints. Our strategy for maintaining feasibility of the iterates is based on a ``retraction" idea adapted from the literature of manifold optimization. When the constraints are convex, we establish the global subsequential convergence of the sequence generated by our algorithm under strict feasibility condition, and analyze its convergence rate when the objective is in addition convex according to the Kurdyka-{\L}ojasiewicz (KL) exponent of the extended objective (i.e., sum of the objective and the indicator function of the constraint set). We also show that the extended objective of a large class of Euclidean norm (and more generally, group LASSO penalty) regularized convex optimization problems is a KL function with exponent $\frac12$; consequently, our algorithm is locally linearly convergent when applied to these problems. We then extend our method to solve DC programs with a single specially structured nonconvex constraint. Finally, we discuss how our algorithms can be applied to solve two concrete optimization problems, namely, group-structured compressed sensing problems with Gaussian measurement noise and compressed sensing problems with Cauchy measurement noise, and illustrate the empirical performance of our algorithms.
\end{abstract}


\section{Introduction}
We consider the following difference-of-convex (DC) optimization problem with smooth inequality and geometric constraints:
\begin{equation}\label{P0}
  \begin{array}{rl}
\min\limits_{x\in\R^n} & P(x) := P_1(x) - P_2(x)\\
{\rm s.t.} & g_i(x)\leq 0,~~ i = 1,\ldots, m,\\
           & x\in C,
  \end{array}
\end{equation}
where $P_1:\R^n\rightarrow\R$ and $P_2:\R^n\rightarrow\R$ are {\color{black}finite-valued convex} functions, $C\subseteq \R^n$ is a nonempty compact convex set, $g_i:\R^n\rightarrow\R$ is smooth with Lipschitz gradient on $C$ for each $i$, and the feasible set $\mathcal{F} :=\{x\in C:\; g_i(x)\le 0, ~ i = 1,\ldots, m\}$ is nonempty.

The model problem (\ref{P0}) covers many important optimization problems arising from diverse areas. For example, one can model the compressed sensing (CS) problem with Cauchy noise by minimizing the difference of $\ell_1$-norm and $\ell_2$-norm \cite{YiLH15} or its variants subject to an inequality constraint defined by the so-called Lorentzian norm \cite{CaBA10}, and the set $C$ in \eqref{P0} can be used for modeling other prior information such as bounds or nonnegativity; more discussions on this particular application can be found in Section~\ref{sec6.2} below. Problem~\eqref{P0} is typically challenging to solve because it is in general nonconvex and nonsmooth.

In this paper, we are interested in developing first-order feasible methods for solving problem \eqref{P0} that rely on affine approximations to the smooth inequality constraints in the subproblems. Our methods utilize only first-order information, and the subproblems that arise are convex constrained optimization problems with the constraint sets being the intersection of $C$ and a polyhedron. 
 Moreover, it also has the desirable feature that  all the iterates generated always belong to the feasible region. In the literature, the most widely used method for solving \eqref{P0} when the objective function $P$ is additionally smooth and also involves affine approximations to the constraint set is the so-called sequential quadratic programming (SQP) method: Indeed, this class of methods typically further requires that the objective function $P$ and all the constraint functions $g_i$'s to be twice continuously differentiable. The underlying strategy of SQP methods is to construct an algorithm which, at each iteration, utilizes a convex quadratic programming problem with affine constraints to approximate the original nonlinear optimization problem. The literature is vast and we refer the readers to \cite{AcDeJa21,Ber95,NoWr06,GiWo12,FlGL02,So09,Au13,BoPau16,WrightTenny04,PaTi93,FlPa08,LawTits01}.
Recent works on SQP-type methods include \cite{Au13,BoPau16}, where the authors proposed and studied the global convergence of some general majorization minimization schemes which cover various SQP-type methods as special cases. However, these SQP-type methods therein are infeasible SQP-type methods, that is, they do not necessarily force all iterates to be feasible. There are also ``feasible SQP methods" that perform arc-search along quadratic curves determined by two directions, namely a descent direction and a feasible direction (see, for example, \cite{PaTi93,FlPa08,LawTits01}). However, these methods make use of the smoothness of the objective and constraint functions to construct the two directions, and it is unclear how their strategies can be adapted to problem~\eqref{P0} that has a possibly nonsmooth objective.

In this paper, we aim at developing feasible first-order methods for the nonconvex nonsmooth problem \eqref{P0} under additional assumptions on $g_i$ by making use of affine approximations for the constraints as in SQP-type methods, and establishing the global convergence of the proposed algorithms. Different from the SQP-type methods covered in \cite{Au13,BoPau16}, our methods maintain feasibility of all iterates for solving \eqref{P0}, under additional assumptions on $g_i$. Our strategy for maintaining feasibility is different from those adopted in the literature of ``feasible SQP methods" that are based on arc-search. Instead, we adapt the retraction strategy used in manifold optimization \cite{ChenMA20,ChenZM19} in our scheme to ensure descent property for the objective function while maintaining the feasibility of the iterates.

{\color{black} For nonconvex optimization problems with difference-of-convex structure, our method belongs to
the so-called successive convex programs strategy which, in each subproblem, a convex approximation to the constraint is formed. In the literature, popular methods of this form include the standard difference-of-convex algorithm (DCA) and its variants (see \cite{DCA18} for a recent survey), and also the moving balls approximation methods \cite{Au13,AuSheTeb10,BoPau16,YuLP20}. For these methods, their subproblems involve quadratic or general convex constraints which are, in general, harder to solve comparing to our proposed method because our method solves a linearly constrained problem in each subproblem. It is also worth mentioning that the technique in using affine approximation for the constraints to construct numerical methods to solve possibly nonconvex problems has been exploited in the literature  (for example, see \cite{Veinott,AckOli16,AcDeJa21}). However, our method differs from these methods in various important aspects. More explicitly, \cite{Veinott,AckOli16} proposed the so-called supporting hyperplane methods which employ a cutting plane strategy. In particular, for the supporting hyperplane methods in \cite{Veinott,AckOli16}, as the algorithm progresses, the number of halfspaces involved in the subproblem {\em grows}, with many of them constructed based on information from the \emph{past iterates}. In contrast, our method only requires a {\it fixed} number of halfspaces in each iteration, and the constraint functions are linearized only at the \emph{current iterate}. {\color{black} Thus, each subproblem in our method bears a similar computational cost. In addition, while the method in \cite{AcDeJa21} can make use of subproblems with a {\em fixed} number of constraints and is applicable to a wider class of problems (e.g., it does not require differentiability of the constraint function), the sequence generated and its cluster points are not guaranteed to be feasible; this is different from our method, which maintains the feasibility of all iterates. Finally,} due to the specific construction of the retraction step, our method ensures the descent property of the objective value: an important aspect that was not exploited in \cite{Veinott,AckOli16,AcDeJa21}.
}

More explicitly, the contributions of this paper are summarized as follows:
\begin{itemize}
  \item We propose a retraction based first-order feasible method that utilizes affine approximations for the constraints in each subproblem for solving problem \eqref{P0} in the case where the constraint functions $g_i$'s are convex under strict feasibility condition, and establish its global subsequential convergence. We also analyze the convergence rate in the fully convex setting, i.e., when $P_2$ is also zero, under suitable Kurdyka-{\L}ojasiewicz (KL) assumptions.
  \item By assuming {\color{black} a stronger version of the strict feasibility condition}, we then extend our method to solve \eqref{P0} with one single specially structured nonconvex constraint function and establish the global subsequential convergence of the extended method. Our assumptions on the constraint function are general enough to allow the choice of the composition of the Lorentzian norm function with an affine function, a commonly used loss function for compressed sensing problems with Cauchy measurement noise \cite{CaBA10}.
  \item Finally, we examine the KL exponents of several structured optimization models of the form \eqref{P0}, an important quantity which quantifies the convergence rate of the proposed method in the convex setting. Specifically, we show that the extended objective\footnote{This refers to the sum of the objective and the indicator function of the constraint set.} of a large class of Euclidean norm (and more generally, group LASSO penalty) regularized convex optimization problems is a KL function with exponent $\frac12$. This complements the recent study in this direction \cite{LiPong18,YuLiPo19,YuLP20,ZYuP20}, and, in particular, gives rise to linear convergence of our method in solving these structured optimization models.
\end{itemize}

The rest of the paper is organized as follows. We introduce notation and preliminary materials in Section~\ref{sec2}. Our retraction based first-order feasible method for solving \eqref{P0} when $g_i$'s are convex is presented and analyzed in Section~\ref{sec3}; we also analyze the convergence rate there in the fully convex setting. In Section~\ref{sec4}, we discuss how our method can be extended to solving \eqref{P0} with one single specially structured nonconvex constraint function. Explicit KL exponents of several structured optimization models in the form of \eqref{P0} are derived in Section~\ref{sec5}. Finally, in Section~\ref{sec6}, we perform numerical experiments to compare the performance of our proposed methods with an adaptation of the method in \cite{Au13} on solving two concrete optimization problems: group-structured CS problems with Gaussian measurement noise and CS problems with Cauchy measurement noise. For the readers' convenience, in the Appendix, we also provide explanations on how the subproblems are solved in our numerical experiments.

\section{Notation and preliminaries}\label{sec2}
In this paper, we use $\R^n$ to denote the Euclidean space of dimension $n$ and $\R^n_+$ to denote the nonnegative orthant of $\R^n$. For two vectors $x$ and $y\in\R^n$, we use $\langle x, y\rangle$ to denote the inner product. The Euclidean norm of $x$ is denoted by $\|x\|$ and the $\ell_1$ norm of $x$ is denoted by $\|x\|_1$. For $x\in\R^n$ and $r\ge 0$, we let $B(x,r)$ denote the closed ball centered at $x$ with radius $r$ {\color{black} with respect to the Euclidean norm}. We also use $\mathbb{N}_+$ to denote the set of positive integers.

For an extended-real-valued function $h: \R^n \to [-\infty, \infty]$, we denote its domain by $\dom h = \{x\in\R^n:\; h(x) <\infty\}$. The function $h$ is said to be proper if $h(x)>-\infty$ for all $x \in \R^n$ and $\dom h\not = \emptyset$. Moreover, a proper function is said to be closed if it is lower semicontinuous. Given a proper closed function $h$, the regular subdifferential $\widehat{\partial} h$ and the (limiting) subdifferential $\partial h$ at $x\in\dom h$ are given respectively by
\[
\widehat{\partial} h(x) = \left\{v\in\R^n:\; \liminf_{y\to x, y\not= x}\frac{h(y) - h(x) - \langle v, y - x\rangle}{\|y - x\|}\geq 0 \right\},
\]
\[
\partial h(x) = \left\{v\in\R^n:\; \exists x^k\overset{h}\rightarrow x\text{ and } v^k\in\widehat{\partial} h(x^k) \text{ with } v^k\rightarrow v \right\},
\]
where $x^k\overset{h}\rightarrow x$ means $x^k\rightarrow x$ and $h(x^k)\rightarrow h(x)$. By convention, if $x\not\in\dom h$, we set $\partial h(x) = \emptyset$. We also write $\dom\partial h:= \{x\in\R^n:\; \partial h(x) \not= \emptyset\}$. When $h$ is continuously differentiable at $x$, we have $\partial h(x) = \{\nabla h(x)\}$ according to \cite[Exercise~8.8(b)]{RoWe98}. Moreover, when $h$ is convex, from \cite[Proposition~8.12]{RoWe98}, the above subdifferentials reduce to the classical subdifferential in convex analysis.

For a nonempty set $C\subseteq\R^n$, the support function is defined as $\sigma_C(x) := \sup_{y\in C}\langle x,y\rangle$ and the indicator function $\delta_C$ is defined as
\[
\delta_C(x): =\begin{cases}
  0 & {\rm if}\ x\in C,\\
  \infty & {\rm otherwise}.
\end{cases}
\]
The normal cone of a closed set $C$ at an $x\in C$ is defined by $\mathcal{N}_C(x):= \partial\delta_C(x)$. In addition, for a nonempty closed set $C\subseteq\R^n$, we denote the distance from a point $x\in\R^n$ to $C$ by $\d(x,C):=\inf\limits_{y\in C}\|x - y\|$.

We next recall the Kurdyka-{\L}ojasiewicz property, which is satisfied by many functions such as proper closed semialgebraic functions, and is important for analyzing global convergence and local convergence rate of first-order methods; see, for example, \cite{AtBo09,AtBR10,AtBS13,BoST14,LiPong18}.
\begin{definition}[{{\bf Kurdyka-{\L}ojasiewicz property and exponent}}]
We say that a proper closed function $h:\R^n\to (-\infty, \infty]$ satisfies the Kurdyka-{\L}ojasiewicz (KL) property at an $\hat{x}\in\dom\partial h$ if there are $a\in(0,\infty]$, a neighborhood $V$ of $\hat{x}$ and a continuous concave function {\color {black} $\phi:[0,a)\to [0, \infty)$ } with {\color {black} $\phi(0) = 0$} such that
\begin{enumerate}[{\rm (i)}]
  \item {\color {black} $\phi$} is continuously differentiable on $(0, a)$ with {\color {black} $\phi' >0$} on $(0, a)$;
  \item for any $x\in V$ with $h(\hat{x}) < h(x) < h(\hat{x}) + a$, it holds that
      \begin{equation}\label{KLdefin}
      {\color {black} \phi}'(h(x) - h(\hat{x}))\d(0, \partial h(x)) \geq 1.
      \end{equation}
\end{enumerate}
If $h$ satisfies the KL property at $\hat{x}\in\dom\partial h$ and the {\color {black} $\phi$} in \eqref{KLdefin} can be chosen as ${\color {black} \phi}(\nu) = a_0\nu^{1 - \alpha}$ for some $a_0 > 0$ and $\alpha\in [0, 1)$, then we say that $h$ satisfies the KL property at $\hat{x}$ with exponent $\alpha$.

A proper closed function $h$ satisfying the KL property at every point in $\dom\partial h$ is called a KL function, and a proper closed function $h$ satisfying the KL property with exponent $\alpha\in[0,1)$ at every point in $\dom\partial h$ is called a KL function with exponent $\alpha$.
\end{definition}

Before ending this section, we recall the following standard constraint qualifications on the constraint set of \eqref{P0} and the first-order necessary optimality conditions for \eqref{P0}.
\begin{definition}[{{\bf MFCQ}}]\label{MFCQ1}
We say that the Mangasarian-Fromovitz constraint qualifications for \eqref{P0} holds at an $x\in \mathcal{F}$ if the following implication holds:
\[
\left.\begin{matrix}
-\sum\limits_{i=1}^m\lambda_i \nabla g_i(x)\in \mathcal{N}_C(x)\\
\lambda_ig_i(x) =0, ~ i = 1, \ldots, m\\
\lambda_i\geq 0, ~ i = 1, \ldots, m
\end{matrix}\right\} \Rightarrow\lambda_i =0, ~ i = 1,\ldots, m.
\]
\end{definition}
{\color{black} It is known that if the $g_i$'s in \eqref{P0} are in addition convex and the (generalized) Slater condition holds, i.e., $\{x \in C: g_i(x) < 0, ~ i = 1, \ldots, m\} \neq \emptyset $, then the MFCQ holds at any point in ${\cal F}$; indeed, suppose that $g_i$'s in \eqref{P0} are in addition convex, $x\in {\cal F}$, $-\sum\limits_{i=1}^m\lambda_i \nabla g_i(x)\in \mathcal{N}_C(x)$, $\lambda_ig_i(x) =0$ and
$\lambda_i\geq 0$ for all $i$, and let $\hat x\in C$ satisfy $\max_{1\le i\le m}g_i(\hat x) \le -\delta < 0$ for some $\delta > 0$. Then
\[
0 \le \sum_{i=1}^m\langle \lambda_i \nabla g_i(x), \hat x - x\rangle\le \sum_{i=1}^m\lambda_i(g_i(\hat x) - g_i(x)) = \sum_{i=1}^m\lambda_ig_i(\hat x)\le -\delta\sum_{i=1}^m\lambda_i,
\]
where the first inequality holds thanks to $-\sum\limits_{i=1}^m\lambda_i \nabla g_i(x)\in \mathcal{N}_C(x)$ and $\hat x\in C$, and the second inequality is due to convexity.
The above display together with $\lambda_i \ge 0$ for all $i$ shows that $\lambda_i = 0$ for all $i$.
}
\begin{definition}[{{\bf {\color{black} Critical point}}}]\label{Stationary}
We say that $x\in\R^n$ is a {\color{black}critical point} of \eqref{P0} if there exists $\lambda = (\lambda_1, \lambda_2, \ldots, \lambda_m)\in\R_+^m$ such that $(x,\lambda)$ satisfies
 \begin{enumerate}[{\rm (i)}]
   \item $g_i(x) \leq 0, ~ i = 1, \ldots, m$, and $x\in C$;
   \item $\lambda_ig_i(x)=0, ~ i = 1, \ldots, m$;
   \item $0\in\partial P_1(x) - \partial P_2(x) + \sum\limits_{i=1}^m\lambda_i \nabla g_i(x) + \mathcal{N}_C(x)$.
 \end{enumerate}
\end{definition}

If the MFCQ holds at every point in $\mathcal{F}$, by using similar arguments as in \cite[Section~2]{YuLP20}, one can show that any local minimizer of \eqref{P0} is a {\color{black}critical point} of \eqref{P0}. We omit its proof for brevity.

\section{Convex constraints with an explicit strict feasible point}\label{sec3}
In this section, we study \eqref{P0} under the following additional assumption:
\begin{assumption}\label{assumption1}
For each $i = 1, \ldots, m$, the function $g_i$ in \eqref{P0} is convex. Moreover, {\color{black} the Slater condition holds and a Slater point is known, i.e., there exists $\xfeas\in C$ with $g_i(\xfeas) < 0$ for $i = 1, \ldots, m$, and $\xfeas$ is explicitly known. }
\end{assumption}
We now describe our method for solving \eqref{P0} under Assumption~\ref{assumption1} in Algorithm~\ref{SQP_retract} below. In essence, our algorithm is an SQP-type method in the sense that, in each iteration, we replace the smooth functions $g_i$ by their affine minorants in the subproblem (see \eqref{subproblem2}). On the other hand, we force our iterates to be feasible by performing a suitable convex combination with the Slater point $\xfeas$ in Step 2a). A line-search scheme is adapted to guarantee sufficient descent in Step 2b). This strategy for maintaining feasibility is different from those adopted in feasible SQP methods, which perform arc-search along quadratic curves \cite{PaTi93,FlPa08}. Our strategy is more similar in spirit to the retraction strategy in manifold optimization \cite{ChenMA20,ChenZM19}: as we shall see in the convergence analysis, they induce similar descent conditions on the objective function values. Since our method makes use of affine approximations (hence, polyhedral) and leverages the idea of retraction, we shall refer to our algorithm as first-order polyhedral approximation (FPA) method with retraction.
\begin{algorithm}
\caption{FPA method with retraction for \eqref{P0} under Assumption~\ref{assumption1}}\label{SQP_retract}
\begin{algorithmic}
\STATE
\begin{description}
  \item[\bf Step 0.] Choose $x^0\in \mathcal{F}$, $0 < \underline{\beta} < \overline{\beta}$, $c > 0$ and $\eta \in (0, 1)$. Choose $\xfeas$ as in Assumption~\ref{assumption1} and set $k = 0$.
  \item[\bf Step 1.] Pick any $\xi^k\in \partial P_2(x^k)$. Choose $\beta_k^0\in[\underline{\beta}, \overline{\beta}]$ arbitrarily and set $\widetilde{\beta} = \beta_k^0$.
  \item[\bf Step 2.] Compute
  \begin{equation}\label{subproblem2}
  \begin{array}{rl}
  \widetilde{u}= \argmin\limits_{x\in\R^n} & P_1(x) - \langle \xi^k, x - x^k\rangle + (2\widetilde{\beta})^{-1}\|x - x^k\|^2\\ [2pt]
      {\rm s.t.}& g_i(x^k) + \langle \nabla g_i(x^k), x - x^k\rangle\leq 0,\ \ i = 1, \ldots, m,\\ [2pt]
      & x\in C.
  \end{array}
  \end{equation}
  \begin{enumerate}[\bf {Step 2}a)]
    \item If $\max\limits_{1 \le i \le m} \{g_i(\widetilde{u})\}\leq 0$, set $\widetilde{x}=\widetilde{u}$ and $\widetilde\tau = 0$; else, find $\widetilde{\tau}\in(0,1)$ such that $\widetilde{x}:=(1 - \widetilde{\tau})\widetilde{u} + \widetilde{\tau} \xfeas$ satisfies $\max\limits_{1 \le i \le m} \{g_i(\widetilde{x})\} = 0$.
    \item If $P(\widetilde{x}) \leq P(x^k) - \frac{c}{2}\|\widetilde{u} - x^k\|^2$, go to \textbf{Step 3}; else, update $\widetilde{\beta}\leftarrow \eta\widetilde{\beta}$ and go to \textbf{Step 2}.
  \end{enumerate}
  \item[\bf Step 3.] Set $x^{k+1} = \widetilde{x}$, $u^k = \widetilde{u}$, $\beta_k = \widetilde{\beta}$, $\tau_k = \widetilde{\tau}$. Update $k \leftarrow k+1$ and go to \textbf{Step 1}.
\end{description}
\end{algorithmic}
\end{algorithm}

We first show that Algorithm~\ref{SQP_retract} is well defined. To this end, it suffices to show that, if $x^k\in {\cal F}$ for some $k\ge 0$, then the subproblem~\eqref{subproblem2} has a unique solution, $\widetilde\tau$ exists in Step 2a) and the line-search step in Step 2b) terminates in finitely many inner iterations, so that $x^{k+1}$ can be generated. Note that such an $x^{k+1}$ belongs to ${\cal F}$ in view of Step 2a). This together with an induction argument would establish the well-definedness of Algorithm~\ref{SQP_retract}.

\begin{theorem}[{{\bf Well-definedness of Algorithm~\ref{SQP_retract}}}]\label{welldef1}
Consider \eqref{P0} and suppose that Assumption~\ref{assumption1} holds. Suppose that $x^k\in {\cal F}$ is {\color{black} generated at the beginning of the $k$-th iteration} of Algorithm~\ref{SQP_retract} for some $k \ge 0$. Then the following statements hold:
\begin{enumerate}[{\rm (i)}]
   \item The convex subproblem \eqref{subproblem2} has a unique solution. Moreover, $\xfeas$ from Assumption~\ref{assumption1} is a Slater point of the feasible set of \eqref{subproblem2}, that is, $\xfeas\in C$ and $$\max\limits_{1 \le i \le m}\{g_i(x^k) + \langle \nabla g_i(x^k), \xfeas - x^k\rangle\} < 0.$$
   \item If $\max\limits_{1 \le i \le m} \{g_i(\widetilde{u})\} > 0$, then there exists $\widetilde{\tau}\in (0, 1)$ such that $\widetilde{x}:=(1 - \widetilde{\tau})\widetilde{u} + \widetilde{\tau} \xfeas$ satisfies $\max\limits_{1 \le i \le m} \{g_i(\widetilde{x})\} = 0$. Moreover, in this case, for any such $\widetilde\tau$, we have
       \begin{equation}\label{taubound0}
       \widetilde{\tau} \leq \frac{L\|\widetilde{u} - x^k\|^2}{-2\max\limits_{1 \le i \le m}\{g_i(\xfeas)\}}\ \ \ {and}\ \ \
       \|\widetilde{x} - \widetilde{u}\| \leq \frac{ML\|\widetilde{u} - x^k\|^2}{-2\max\limits_{1 \le i \le m}\{g_i(\xfeas)\}},
       \end{equation}
       where $M := \sup\limits_{x\in C}\|x - \xfeas\|$ and $L := \max\limits_{1 \le i \le m} L_{g_i}$, with $L_{g_i}$ being the Lipschitz continuity modulus of $\nabla g_i$ on $C$ for each $i$.
   \item Step 2 of algorithm \ref{SQP_retract} terminates in finitely many inner iterations.
 \end{enumerate}
 Thus, an $x^{k+1}\in {\cal F}$ can be generated at the end of the $(k+1)$-th iteration of Algorithm~\ref{SQP_retract}.
\end{theorem}
\begin{proof}
(i): First, we note that the feasible region of \eqref{subproblem2} is closed and convex. Next, we show that this set is nonempty by showing that it contains the point $\xfeas\in C$. Indeed, since each $g_i$ is convex, we have that for each $i$, $g_i(y) + \langle \nabla g_i(y),\xfeas - y\rangle\leq g_i(\xfeas) < 0$, for all $y\in \R^n$. This also shows that $\xfeas$ is a Slater point of the feasible set of \eqref{subproblem2}. Now, since $P_1$ is convex and continuous, we see that $x\mapsto P_1(x) - \langle \xi^k, x - x^k\rangle + \frac{1}{2\widetilde{\beta}}\|x - x^k\|^2$ is strongly convex and continuous for any $\widetilde{\beta} > 0$. Therefore, \eqref{subproblem2} is a convex optimization problem with a strongly convex objective function over a nonempty closed convex set, and hence, has a unique solution.

(ii): Suppose that $\max\limits_{1 \le i \le m} \{g_i(\widetilde{u})\} > 0$. By Assumption \ref{assumption1}, we have $\max\limits_{1 \le i \le m} \{g_i(\xfeas)\} < 0$. Let $g(x)= \max\limits_{1 \le i \le m} g_i(x)$. Then, by the continuity of $g$, there exists $\widetilde{\tau}\in (0, 1)$ such that $\widetilde{x} := (1 - \widetilde{\tau})\widetilde{u} + \widetilde{\tau} \xfeas$ satisfies $g( \widetilde{x})=\max\limits_{1 \le i \le m} \{g_i(\widetilde{x})\} = 0$.

Since $g$ is convex, we have
\[
0 = g(\widetilde{x})=\max\limits_{1 \le i \le m} \{g_i(\widetilde{x})\} \leq (1 - \widetilde{\tau})\max\limits_{1 \le i \le m} \{g_i(\widetilde{u})\} + \widetilde{\tau} \max\limits_{1 \le i \le m} \{g_i(\xfeas)\}.
\]
Upon rearranging terms in the above inequality, we see that
\begin{equation}\label{taubound_1}
\begin{aligned}
\widetilde{\tau}&\overset{\rm (a)}\leq\frac{\max\limits_{1 \le i \le m} \{g_i(\widetilde{u})\}}{\max\limits_{1 \le i \le m} \{g_i(\widetilde{u})\} - \max\limits_{1 \le i \le m}\{g_i(\xfeas)\}} \overset{\rm (b)}\leq \frac{\max\limits_{1 \le i \le m}\{g_i(\widetilde{u})\}}{-\max\limits_{1 \le i \le m}\{g_i(\xfeas)\}} \\
& \overset{\rm (c)}\le \frac{\max\limits_{1 \le i \le m}\left\{g_i(x^k) + \langle\nabla g_i(x^k),\widetilde{u}-x^k\rangle + (L_{g_i}/2)\|\widetilde{u} - x^k\|^2\right\}}{-\max\limits_{1 \le i \le m}\{g_i(\xfeas)\}} \overset{\rm (d)}\leq \frac{L\|\widetilde{u} - x^k\|^2}{-2\max\limits_{1 \le i \le m}\{g_i(\xfeas)\}},
\end{aligned}
\end{equation}
where (a) and (b) hold because $\max\limits_{1 \le i \le m}\{g_i(\widetilde{u})\} > 0$ and $\max\limits_{1 \le i \le m}\{g_i(\xfeas)\} < 0$, (c) is true because $\nabla g_i$ is Lipschitz with modulus $L_{g_i}$ on the convex set $C$ that contains both $\widetilde u$ and $x^k$, and (d) made use of $L := \max\limits_{1 \le i \le m} L_{g_i}$ and the fact that $\widetilde u$ is feasible for \eqref{subproblem2}. Hence,
\[
\begin{aligned}
\|\widetilde{x} - \widetilde{u}\|& = \widetilde{\tau}\|\widetilde{u} - \xfeas\|\overset{\rm (a)}\leq \widetilde{\tau}M \overset{\rm (b)}\leq \frac{ML}{-2\max\limits_{1 \le i \le m}\{g_i(\xfeas)\}}\|\widetilde{u} - x^k\|^2,
\end{aligned}
\]
where $\rm (a)$ holds because $\widetilde{u}\in C$ and $M = \sup\limits_{x\in C}\|x - \xfeas\|$, and $\rm(b)$ follows from \eqref{taubound_1}.

(iii): Note that we have $\widetilde{x} = (1 - \widetilde{\tau})\widetilde{u} + \widetilde{\tau} \xfeas$ for some $\widetilde{\tau}\in [0,1]$, and hence
\[
\begin{aligned}
P(\widetilde{x}) &= P_1(\widetilde{x}) - P_2(\widetilde{x}) \overset{\rm (a)}\leq P_1(\widetilde{x}) - P_2(x^k) - \langle \xi^k, \widetilde{x} - x^k\rangle \\
&\overset{\rm (b)}\leq (1 - \widetilde{\tau})P_1(\widetilde{u}) + \widetilde{\tau}P_1(\xfeas) - P_2(x^k) - (1 - \widetilde{\tau})\langle \xi^k, \widetilde{u} - x^k\rangle - \widetilde{\tau}\langle \xi^k, \xfeas - x^k\rangle\\
&= P_1(\widetilde{u}) - \langle \xi^k, \widetilde{u} - x^k\rangle + \widetilde{\tau}\big(P_1(\xfeas) - P_1(\widetilde{u}) - \langle \xi^k, \xfeas - \widetilde{u}\rangle \big) - P_2(x^k)\\
&\overset{\rm (c)}\leq P_1(x^k) - \frac{1}{2\widetilde{\beta}}\|\widetilde{u} - x^k\|^2  - P_2(x^k) + \widetilde{\tau}\big(P_1(\xfeas) - P_1(\widetilde{u}) - \langle \xi^k, \xfeas - \widetilde{u}\rangle \big)\\
&\overset{\rm (d)}\leq P(x^k) - \frac{1}{2\widetilde{\beta}}\|\widetilde{u} - x^k\|^2 + \widetilde{\tau}M_1 \overset{\rm (e)}\leq P(x^k) - \left(\frac{1}{2\widetilde{\beta}} - \frac{M_1L}{-2\max\limits_{1\le i\le m}\{g_i(\xfeas)\}}\right)\|\widetilde{u} - x^k\|^2,
\end{aligned}
\]
where (a) follows from the convexity of $P_2$, (b) holds because $P_1$ is convex, (c) follows from the facts that $\widetilde{u}$ is the minimizer of \eqref{subproblem2} and $x^k$ is feasible for \eqref{subproblem2} (since $x^k \in {\cal F}$), (d) holds with $M_1:= \sup\limits_{y,z\in C}\sup\limits_{\xi\in\partial P_2(z)} \{P_1(\xfeas) - P_1(y) - \langle\xi, \xfeas - y\rangle\}\in [0,\infty)$ (note that $M_1 < \infty$ thanks to \cite[Theorem~2.6]{Tu98} and the compactness of $C$, and $M_1 \ge 0$ as $\xfeas\in C$), and (e) follows from \eqref{taubound0}.

Therefore, as long as $\widetilde{\beta} < \frac{1}{2}\left(\frac{c}2 + \frac{M_1L}{-2\max\limits_{1 \le i \le m}\{g_i(\xfeas)\}}\right)^{-1}$ (independent of $k$), we have $P(\widetilde{x}) \leq P(x^k) - \frac{c}{2}\|\widetilde{u} - x^k\|^2$. In view of {Step 2b)} of Algorithm \ref{SQP_retract}, we conclude that the inner loop terminates finitely. This completes the proof.
\end{proof}

\begin{remark}[{{\bf Uniform lower bound on $\beta_k$}}]\label{gamma}
From the proof of Theorem~\ref{welldef1}(iii), we see that for the $\beta_k$ in Algorithm~\ref{SQP_retract}, it holds that
\[\inf_k \beta_k \ge \betamin:= \textstyle\min\left\{\frac{\eta}{2}\left(\frac{c}2 + \frac{M_1L}{-2\max\limits_{1 \le i \le m}\{g_i(\xfeas)\}}\right)^{-1},\underline{\beta}\right\},\]
where $\eta$, $c$ and $\underline{\beta}$ are given in Step 0 of Algorithm~\ref{SQP_retract}, $L$ is given in Theorem~\ref{welldef1}(ii) and $M_1= \sup\limits_{y,z\in C}\sup\limits_{\xi\in\partial P_2(z)} \{P_1(\xfeas) - P_1(y) - \langle\xi, \xfeas - y\rangle\}$.
\end{remark}

\begin{remark}[{{\bf Finding $\widetilde\tau$}}]
  In Step 2a) of Algorithm~\ref{SQP_retract}, when $\max\limits_{1 \le i \le m} \{g_i(\widetilde u)\} > 0$, we need to find $\widetilde\tau\in(0,1)$ so that $\max\limits_{1 \le i \le m} \{g_i((1-\widetilde\tau)\widetilde u + \widetilde\tau \xfeas)\} = 0$; such a $\widetilde\tau$ exists thanks to Theorem~\ref{welldef1}(ii). In fact, we only need to find $\widetilde\tau_i\in(0,1)$ so that $g_i((1-\widetilde\tau_i)\widetilde u + \widetilde\tau_i \xfeas) = 0$ for each $i\in I:= \{i: g_i(\widetilde{u}) > 0\}$. Then taking $\widetilde{\tau} = \max\limits_{i\in I}\{\widetilde{\tau}_i\}$ and letting $\widetilde{x} := (1 - \widetilde{\tau})\widetilde{u} + \widetilde{\tau} \xfeas$, it can be verified that $\max\limits_{1 \le i \le m} \{g_i(\widetilde{x})\} = 0$. Since $g_i$ is convex and smooth,  $\widetilde\tau_i$ can be found efficiently by solving the equation $g_i((1-t)\widetilde u + t \xfeas) = 0$ via the Newton's method, starting with $t = 0$. Moreover, in the case where $g_i$ is a quadratic function (e.g., when $g_i$ is the least squares loss function), $\widetilde\tau_i$ admits a closed form solution.
\end{remark}

\begin{remark}\label{KKT}
Since the Slater condition holds for the convex subproblem \eqref{subproblem2} according to Theorem~\ref{welldef1}(i), from \cite[Corollary~28.2.1,
Theorem~28.3]{Ro70}, we have that for each $k$, $\widetilde{u}$ is a minimizer of \eqref{subproblem2} if and only if
there exists $\widetilde\lambda = (\widetilde\lambda_1,\ldots,\widetilde\lambda_m)\in \R_+^m$ (which is called a Lagrange multiplier) such that the following conditions hold:
\begin{enumerate}[{\rm (i)}]
  \item  $g_i(x^k) + \langle \nabla g_i(x^k), \widetilde{u} - x^k\rangle \leq 0$, $i = 1, \ldots m$, and $\widetilde{u}\in C$;
  \item  $\widetilde\lambda_i(g_i(x^k) + \langle \nabla g_i(x^k), \widetilde{u} - x^k\rangle)=0$, $i = 1, \ldots m$;
  \item  $0\in\partial P_1(\widetilde{u}) - \xi^k + \widetilde\beta^{-1}(\widetilde{u} - x^k) + \sum\limits_{i=1}^m\widetilde\lambda_i\nabla g_i(x^k) + \mathcal{N}_C(\widetilde{u})$.
\end{enumerate}
\end{remark}

We next show that any cluster point of the sequence $\{x^k\}$ generated by Algorithm~\ref{SQP_retract} is a {\color {black} critical point} of \eqref{P0} in the sense of Definition~\ref{Stationary}.
\begin{theorem}[{{\bf Subsequential convergence of Algorithm~\ref{SQP_retract}}}]\label{convergence2}
Consider \eqref{P0} and suppose that Assumption~\ref{assumption1} holds. Let {\color{black}$\{(x^k, u^k)\}$} be the sequence generated by Algorithm \ref{SQP_retract} and $\lambda^k$ be a Lagrange multiplier of \eqref{subproblem2} with $\widetilde{\beta} = \beta_k$. Then the following statements hold:
\begin{enumerate}[{\rm(i)}]
  \item It holds that $\lim\limits_{k\rightarrow\infty}\|u^k - x^k\| = 0$.
  \item The sequences $\{x^k\}$ and $\{\lambda^k\}$ are bounded.
  \item Any accumulation point of $\{x^k\}$ is a {\color{black}critical point} of \eqref{P0}.
\end{enumerate}
\end{theorem}

\begin{proof}
Using Step 2b), we obtain $P(x^{k+1}) \leq P(x^k) - \frac{c}{2}\|u^k - x^k\|^2$. Rearranging terms and summing both sides of the inequality from zero to infinity, we have
\[
\frac{c}{2}\sum_{k=0}^\infty\|u^k - x^k\|^2\leq P(x_0) - \lim\limits_{k\rightarrow\infty} P(x^k)\leq P(x_0) - \inf\limits_{x\in C} P(x) < \infty,
\]
where the second inequality holds because $\{x^k\}\subseteq C$ by construction, and the finiteness follows from the compactness of $C$ and the continuity of $P$. Therefore, $\|u^k - x^k\|\rightarrow 0$. This proves (i).

We now prove (ii). Since $\{x^k\}\subseteq C$ and $C$ is compact, $\{x^k\}$ is bounded. We next show that $\{\lambda^k\}$ is bounded. Suppose to the contrary that $\{\lambda^k\}$ is unbounded and let $\{\lambda^{k_j}\}$ be a subsequence of $\{\lambda^k\}$ such that $\|\lambda^{k_j}\|\rightarrow \infty$. In view of the compactness of $C$, by passing to a further subsequence if necessary, we may also assume that $\lim\limits_{j\rightarrow\infty}x^{k_j} = x^*$ for some $x^*\in C$, and $\lim\limits_{j\rightarrow\infty}\frac{\lambda^{k_j}}{\|\lambda^{k_j}\|} = \lambda^*$ for some $\lambda^*\in\R^m_+$ with $\|\lambda^*\| = 1$.

From \eqref{subproblem2} and the definition of Lagrange multipliers (see Remark~\ref{KKT}), we have
\begin{equation}\label{KKT1}
0\in \partial P_1(u^{k_j}) -\xi^{k_j} + \frac{1}{\beta_{k_j}}(u^{k_j} - x^{k_j}) + \sum\limits_{i=1}^m\lambda^{k_j}_i\nabla g_i(x^{k_j}) +\mathcal{N}_C(u^{k_j}),
\end{equation}
and
\begin{equation}\label{KKT2}
\lambda^{k_j}_i (g_i(x^{k_j}) + \langle \nabla g_i(x^{k_j}), u^{k_j} - x^{k_j}\rangle) = 0,\ \ \  i = 1, \ldots, m.
\end{equation}
Also, we note from Remark~\ref{gamma}, the compactness of $C$ and the continuity of $P_2$ that the sequences $\{x^k\}$, $\{u^k\}$, $\{1/\beta_{k}\}$, $\{\xi^{k}\}$ are bounded. Furthermore, by \cite[Theorem~24.7]{Ro70}, the sets $\{\partial P_1(u^{k_j})\}$ are uniformly bounded thanks to the continuity of $P_1$. Now, dividing $\|\lambda^{k_j}\|$ on both sides of \eqref{KKT1} and \eqref{KKT2} and passing to the limit as $j$ goes to infinity, we have upon recalling $\lim\limits_{j\rightarrow\infty}x^{k_j} = x^*$ and item (i), and invoking the closedness of ${\cal N}_C$ that
\begin{equation*}
0\in \sum\limits_{i=1}^m \lambda^*_i\nabla g_i(x^*) + \mathcal{N}_C(x^*)\ \ {\rm and}\ \ \lambda^*_ig_i(x^*) = 0,\ \ \ i = 1, \ldots, m.
\end{equation*}
Since $\lambda^*\in\R^m_+$, the above relation together with Assumption~\ref{assumption1} (which implies the MFCQ in Definition~\ref{MFCQ1}) and the fact that $\|\lambda^*\|=1$ leads to a contradiction. Thus, $\{\lambda^k\}$ is bounded.

Finally, we prove (iii). Suppose $\bar{x}$ is an accumulation point of $\{x^k\}$ with $\lim\limits_{j\rightarrow\infty}x^{k_j} = \bar{x}$ for some convergent subsequence $\{x^{k_j}\}$. Since $\{\lambda^k\}$ is bounded thanks to item (ii), passing to a further subsequence if necessary, we may assume without loss of generality that $\lim\limits_{j\rightarrow\infty}\lambda^{k_j} = \bar{\lambda}$ for some $\bar\lambda\in \R^m_+$. Since $\|u^k - x^k\|\rightarrow 0$ according to item (i), we also have $\lim\limits_{j\rightarrow\infty}u^{k_j} = \bar{x}$.
Using these together with the closedness of $\partial P_1$, $\partial P_2$ and ${\cal N}_C$, the uniform boundedness of the sequences of sets $\{\partial P_1(u^{k_j})\}$ and $\{\partial P_2(x^{k_j})\}$, the Lipschitz continuity of $\nabla g$ and the boundedness of $\{1/\beta_k\}$ (see Remark~\ref{gamma}), we have upon passing to the limit as $j$ goes to infinity in item (iii) of Remark~\ref{KKT} with $(u^{k_j},x^{k_j},\xi^{k_j},\lambda^{k_j})$ in place of $(\widetilde u,x^k,\xi^k,\widetilde\lambda)$ that
\[
0\in \partial P_1(\bar{x}) - \partial P_2(\bar{x}) + \sum\limits_{i=1}^m\bar{\lambda}_i\nabla g_i(\bar{x}) + \mathcal{N}_C(\bar{x}).
\]
Similarly, using item (ii) of Remark~\ref{KKT} with $(u^{k_j},x^{k_j},\lambda^{k_j})$ in place of $(\widetilde u,x^k,\widetilde\lambda)$ and the continuity of $\nabla g$, we have upon passing to the limit that
\[
\bar{\lambda}_i g_i(\bar{x}) = 0,\ \ i = 1, \ldots, m.
\]
Finally, since for each $i$, $g_i$ is continuous and $g_i(x^k)\leq0$ thanks to Step 2a), we have $g_i(\bar{x})\leq0$. In addition, since $\{x^k\}\subseteq C$ and $C$ is closed, we have $\bar{x}\in C$. Hence $\bar{x}$ is a {\color {black} critical point} of \eqref{P0}. This completes the proof.
\end{proof}

\subsection{Convergence rate analysis in convex setting}
In this subsection, we study the global convergence of the sequence $\{x_k\}$ generated by Algorithm \ref{SQP_retract} and its convergence rate, under the following fully convex setting:
\begin{assumption}\label{assumption3}
  Suppose that in \eqref{P0}, $P_2 = 0$.
\end{assumption}

\begin{theorem}[{{\bf Convergence rate of Algorithm \ref{SQP_retract} in convex setting}}]\label{localin}
Consider \eqref{P0} and suppose that Assumptions \ref{assumption1} and \ref{assumption3} hold. Let $\{x^k\}$ be the sequence generated by Algorithm \ref{SQP_retract}. Then $\{x^k\}$ converges to a minimizer $x^*$ of \eqref{P0}. If in addition $x\mapsto P_1(x) + \sum\limits_{i=1}^m\delta_{g_i(\cdot)\leq 0}(x) + \delta_C(x)$ is a KL function with exponent $\alpha\in [0, 1)$, then the following statements hold:
\begin{enumerate}[{\rm (i)}]
   \item If $\alpha\in[0, \frac{1}{2}]$, then there exist $c_0 > 0$, $s\in (0,1)$ and $k_0\in \mathbb{N_+}$, such that
       \[
       \|x^k - x^*\|\leq c_0 s^k ~~\text{  for  } k > k_0.
       \]
   \item If $\alpha\in(\frac{1}{2}, 1)$, then there exist $c_0 > 0$, and $k_0\in \mathbb{N_+}$, such that
       \[
       \|x^k - x^*\|\leq c_0k^{-\frac{1-\alpha}{2\alpha-1}} ~~\text{ for } k > k_0.
       \]
 \end{enumerate}
\end{theorem}
\begin{proof}
Since $C$ is a compact set and $g$ is continuous, the feasible set ${\cal F}$ of \eqref{P0} is compact. Combining this with the fact that $P_1$ is continuous, we have $S:= \Argmin\limits_{x\in {\cal F}} P_1(x) \not=\emptyset$. Let $\Upsilon$ be the set of accumulation point of $\{x^k\}$ for notational simplicity. Using Theorem \ref{convergence2} and \cite[Theorem~28.3]{Ro70}, we see that
\[
\emptyset\not=\Upsilon\subseteq S.
\]

From Remark~\ref{KKT}, Assumptions~\ref{assumption1} and \ref{assumption3}, and the convexity of $P_1$, we see that for each $k$, there exists $\lambda^k \in \R^m_+$ such that $\lambda^k_i(g_i(x^k) + \langle \nabla g_i(x^k), u^k - x^k\rangle) = 0$ for all $i = 1, \ldots, m$ and $u^k$ is a minimizer of the following function:
\[
L_k(x):= P_1(x) + \delta_C(x) + \frac{1}{2\beta_k}\|x - x^k\|^2 + \sum\limits_{i = 1}^m\lambda^k_i(g_i(x^k) + \langle \nabla g_i(x^k), x - x^k\rangle).
\]
Note that $x\mapsto L_k(x)$ is strongly convex with modulus $\frac1{\beta_k}$. Then we see that for any $x\in C$,
\begin{equation}\label{lagrange}
\begin{aligned}
&P_1(u^k) + \frac{1}{2\beta_k}\|u^k - x^k\|^2= L_k(u^k)\leq L_k(x) - \frac{1}{2\beta_k}\|x - u^k\|^2\\
&= P_1(x) + \frac{1}{2\beta_k}\|x - x^k\|^2 + \sum\limits_{i = 1}^m\lambda^k_i(g_i(x^k) + \langle \nabla g_i(x^k), x - x^k\rangle) - \frac{1}{2\beta_k}\|x - u^k\|^2,
\end{aligned}
\end{equation}
where the first equality holds because $\lambda^k_i(g_i(x^k) + \langle \nabla g_i(x^k), u^k - x^k\rangle) = 0$ for $i = 1, \ldots, m$.

Let $F(x):= P_1(x) + \delta_C(x) + \sum\limits_i\delta_{g_i(\cdot) \le0}(x)$ for notational simplicity. For any $\bar{x}\in S$, from Step 2a) and Step 3 of Algorithm \ref{SQP_retract}, we have $\tau_k\in[0, 1]$, $x^{k+1} = (1-\tau_k)u_k +\tau_k\xfeas \in \mathcal{F}$ and

\begin{align}\label{F}
&F(x^{k+1})= P_1(x^{k+1})\overset{\rm (a)}\leq P_1(u^k) + \tau_k(P_1(\xfeas) - P_1(u^k)) + \frac{1}{2\beta_k}\|u^k - x^k\|^2 \nonumber\\
&\overset{\rm (b)}\leq P_1(\bar{x}) + \frac{1}{2\beta_k}\|\bar{x} - x^k\|^2 + \sum_{i=1}^m\lambda^k_i(g_i(x^k) + \langle \nabla g_i(x^k), \bar{x} - x^k\rangle)- \frac{1}{2\beta_k}\|\bar{x} - u^k\|^2 \nonumber\\
 &~~~~+ \tau_k(P_1(\xfeas) - P_1(u^k))\nonumber\\
 &\overset{\rm (c)}\leq P_1(\bar{x}) + \frac{1}{2\beta_k}\|\bar{x} - x^k\|^2 - \frac{1}{2\beta_k}\|\bar{x} - u^k\|^2 + \tau_k(P_1(\xfeas) - P_1(u^k))\nonumber\\
 &\overset{\rm (d)}\leq P_1(\bar{x}) + \frac{1}{2\beta_k}\|\bar{x} - x^k\|^2 - \frac{1}{2\beta_k}\|\bar{x} - x^{k+1}\|^2 + M_2\tau_k + \frac{1}{2\beta_k}(\|\bar{x} - x^{k+1}\|^2 - \|\bar{x} - u^k\|^2),
 \end{align}
where (a) holds because $P_1$ is convex, (b) follows from \eqref{lagrange} with $x = \bar x$, (c) holds because $\lambda^k_i\geq 0$ and $g_i(x^k) + \langle \nabla g_i(x^k), \bar{x} - x^k\rangle \leq g_i(\bar{x}) \leq 0$ (thanks to $\bar{x}\in S \subseteq\mathcal{F}$), (d) holds with $$M_2:= P_1(\xfeas)- \inf\limits_{x\in C}P_1(x) \in [0,\infty),$$ where the finiteness follows from the continuity of $P_1$ and the compactness of $C$.

We now derive an upper bound for \eqref{F}. Notice that
\begin{align}\label{normsqu}
\|\bar{x} - x^{k+1}\|^2 - \|\bar{x} - u^k\|^2& =\left(\|\bar{x} - x^{k+1}\| - \|\bar{x} - u^k\|\right)\left(\|\bar{x} - x^{k+1}\| + \|\bar{x} - u^k\|\right) \nonumber\\
&\overset{\rm (a)}\leq \|x^{k+1} - u^k\|\left(\|\bar{x} - x^{k+1}\| + \|\bar{x} - u^k\|\right)\nonumber\\
&\overset{\rm (b)}= \tau_k\|\xfeas - u^k\|\left(\|\bar{x} - x^{k+1}\| + \|\bar{x} - u^k\|\right)\overset{\rm (c)}\leq 2M_3^2 \tau_k,
\end{align}
where (a) follows from the triangle inequality, (b) holds because $x^{k+1} - u^k = \tau_k(\xfeas - u^k)$ and $\tau_k \geq 0$, (c) holds with $M_3:= \sup\limits_{x,y\in C}\|x - y\| < \infty$, where the finiteness is due to the compactness of $C$.

Now combining \eqref{F} and \eqref{normsqu}, we have
\begin{align}
F(x^{k+1})&\leq P_1(\bar{x}) + \frac{1}{2\beta_k}\|\bar{x} - x^k\|^2 - \frac{1}{2\beta_k}\|\bar{x} - x^{k+1}\|^2 + \left(M_2 + \frac{M_3^2}{\beta_k}\right)\tau_k \nonumber\\
&\overset{\rm (a)}\leq P_1(\bar{x}) + \frac{1}{2\beta_k}\|\bar{x} - x^k\|^2 - \frac{1}{2\beta_k}\|\bar{x} - x^{k+1}\|^2 + \left(M_2 + \frac{M_3^2}{\beta_k}\right)\frac{L\|u^k - x^k\|^2}{-2\max\limits_{1 \le i \le m}\{g_i(\xfeas)\}}\nonumber\\
&\overset{\rm (b)}\leq P_1(\bar{x}) + \frac{1}{2\beta_k}\|\bar{x} - x^k\|^2 - \frac{1}{2\beta_k}\|\bar{x} - x^{k+1}\|^2 + M_4(P_1(x^k) - P_1(x^{k+1})), \nonumber
\end{align}
where (a) follows from \eqref{taubound0}, (b) holds because $\|u^k - x^k\|^2 \leq \frac{2}{c}(P(x^k) - P(x^{k+1}))$ from Step 2b) of Algorithm~\ref{SQP_retract} and we define $M_4 := \left(M_2 + \frac{M_3^2}{\betamin}\right)\frac{L}{-c\max\limits_{1 \le i \le m}\{g_i(\xfeas)\}}\in (0,\infty)$, with $\betamin$ given in Remark~\ref{gamma}.

Rearranging terms in the above inequality and noting $\bar{F}^* := \inf F = P_1(\bar{x})$ whenever $\bar{x}\in S$, we have
\[
\beta_k(F(x^{k+1}) - \bar{F}^*)\leq \frac{1}{2}\|\bar{x} - x^k\|^2 - \frac{1}{2}\|\bar{x} - x^{k+1}\|^2 + M_4\beta_k(F(x^k) - F(x^{k+1})).
\]
Recall that $\betamin\leq\beta_k \leq\bar{\beta}$, where $\betamin$ is defined in Remark \ref{gamma} and $\bar{\beta}$ is a parameter in Step 0 of Algorithm \ref{SQP_retract}. Then for any $\bar{x} \in S$, we have
\[
\betamin(F(x^{k+1}) - \bar{F}^*)\leq \frac{1}{2}\|\bar{x} - x^k\|^2 - \frac{1}{2}\|\bar{x} - x^{k+1}\|^2 + M_4\bar{\beta}(F(x^k) - F(x^{k+1})).
\]
Rearranging terms in the above inequality, we have
\begin{equation*}
(\betamin + M_4\bar{\beta})(F(x^{k+1}) - \bar{F}^*)\leq \frac{1}{2}\|\bar{x} - x^k\|^2 - \frac{1}{2}\|\bar{x} - x^{k+1}\|^2 + M_4\bar{\beta}(F(x^k) - \bar{F}^*).
\end{equation*}
The desired results now follow from the same arguments as in the proof of \cite[Theorem~3.12]{YuLP20}, starting from \cite[Eq.~(3.16)]{YuLP20}. This completes the proof.
\end{proof}


\section{Extension: Structured non-convex constraint with strong strict feasibility}\label{sec4}

In this section, we consider DC optimization problem with structured nonconvex inequality constraint and geometric constraint. As we will see later, our setting covers, in particular, several important model problems such as the DC optimization problem with Lorentzian norm and box constraints. In order to simplify notation, we only consider the case where one single nonconvex inequality constraint is present in \eqref{P0}.

More explicitly, we consider the following special case of \eqref{P0}:
\begin{equation}\label{P1}
  \begin{array}{rl}
\min\limits_{x\in\R^n} & P(x) := P_1(x) - P_2(x)\\
{\rm s.t.} & \ell(Ax - b)\leq \sigma,\\
           & x\in C,
  \end{array}
\end{equation}
where $P_1$ and $P_2$ are {\color{black}finite-valued convex} functions, $C$ is a nonempty compact convex set, $A\in\R^{p\times n}$, $b\in\R^p$,  $\sigma >0$, the set ${\cal F}_1:= \{x\in C:\; \ell(Ax - b)\le \sigma\}\neq \emptyset$, and $\ell$ also satisfies the following Assumption~\ref{gassum}:
\begin{assumption}\label{gassum}
The function $\ell:\R^p\to \R$ is defined as
\[
\ell(u):= \sum_{i=1}^{p} {\color{black} \varphi}(u_i^2),
\]
where ${\color{black} \varphi}:\R_+\rightarrow \R_+$ satisfies the following conditions:
\begin{enumerate}[{\rm (i)}]
  \item ${\color{black} \varphi}$ is a continuous concave function with ${\color{black} \varphi}(0) = 0$;
  \item ${\color{black} \varphi}$ is continuously differentiable with positive derivative on $(0, \infty)$;
  \item The limit $\vartheta:= \lim\limits_{t\downarrow 0}{\color{black} \varphi}'(t)$ exists, and the function $t\mapsto {\color{black} \varphi}'_+(t)$ is locally Lipschitz on $[0,\infty)$.\footnote{Here, and throughout this paper, we use ${\color{black} \varphi}'_+(t)$ to denote the right hand derivative of ${\color{black} \varphi}$ at $t$, which is defined as ${\color{black} \varphi}'_+(t)=\lim\limits_{h \rightarrow 0^+}\frac{{\color{black} \varphi}(t+h)-{\color{black} \varphi}(t)}{h}$.}
\end{enumerate}
\end{assumption}

We present below some concrete examples of ${\color{black} \varphi}$ that satisfies the conditions in Assumption~\ref{gassum}.
\begin{example}
As in \cite{AhPX17}, for a given $a > 0$, we consider ${\color{black} \varphi}:\R_+\rightarrow \R_+$ given as follows:
\[
{\color{black} \varphi}(u) = \frac{(a + 1)u}{a + u}.
\]
One can see that ${\color{black} \varphi}$ is a continuous concave function and is differentiable with positive derivative on $(0, \infty)$ as its first and second derivatives exist:
\[
{\color{black} \varphi}'(u) = \frac{a(a+1)}{(a+u)^2}\ \text{ and }\  {\color{black} \varphi}''(u) = -\frac{2a(a+1)}{(a+u)^3}.
\]
From this we see further that ${\color{black} \varphi}'$ is Lipschitz continuous on $[0,\infty)$ with modulus $\frac{2(a+1)}{a^2}$. Thus ${\color{black} \varphi}$ satisfies the conditions in Assumption \ref{gassum}.
\end{example}

\begin{example}\label{Loren}
We consider the function ${\color{black} \varphi}:\R_+\to \R$ defined as follows for a given $\gamma$ \cite{CaWB08}:
\begin{equation*}
{\color{black} \varphi}(u) = \log(u + \gamma^2) - \log(\gamma^2).
\end{equation*}
One can see that ${\color{black} \varphi}$ satisfies the conditions in Assumption \ref{gassum} and ${\color{black} \varphi}'$ is Lipschitz continuous on $[0,\infty)$ with modulus $\frac{1}{\gamma^4}$. In this case, we have
\[
\ell(u) = \sum_{i=1}^p \log\left(1 + \frac{u_i^2}{\gamma^2}\right),
\]
which is the Lorentzian norm: This function is used as the loss function for measurement noise following the Cauchy distribution; see \cite{CaBA10,CaRA16}.
\end{example}
Recall that the Slater condition (i.e., Assumption~\ref{assumption1}) was instrumental in deriving the first-order optimality conditions of \eqref{P0} and analyzing convergence of Algorithm~\ref{SQP_retract}.
Here, for analyzing \eqref{P1} and developing an analogue of Algorithm~\ref{SQP_retract} for solving it, we will consider the following stronger strict feasibility assumption on the feasible set ${\cal F}_1$ of \eqref{P1}.
\begin{assumption}\label{assumption4}
{\color{black} There exists $\xfeasss\in C$ such that $A\xfeasss = b$ and $\xfeasss$ is explicitly known.}
\end{assumption}

Before presenting our algorithm for solving \eqref{P1} under Assumptions~\ref{gassum} and \ref{assumption4}, to simplify the presentation of the algorithm, we first introduce some auxiliary notation.

Consider \eqref{P1} and suppose Assumption~\ref{gassum} holds. For any $y\in\R^n$, let
\begin{equation}\label{defomega}
\omega_i^y := {\color{black} \varphi}_+'\big((a_i^T y - b_i)^2\big), ~i= 1,\ldots, p,
\end{equation}
\begin{equation}\label{defsigma}
\tilde{\sigma}^y := \sigma - \ell(Ay - b) + \sum_{i=1}^p \omega_i^y(a_i^Ty - b_i)^2,
\end{equation}
and define the function $\ell^y:\R^p\to \R$ by
\begin{equation}\label{defell}
\ell^y(u):= \sum_{i=1}^p \omega_i^yu_i^2,
\end{equation}
where $A$ and $b$ are as in \eqref{P1} and $a_i^T$ denotes the $i$th row of $A$.
Note that under Assumption~\ref{gassum}, the function ${\color{black} \varphi}$ is concave. Hence, for any $x,y\in\R^n$ and each $i=1,\ldots,p$, we have from the super-gradient inequality that
\[
{\color{black} \varphi}\big((a_i^T x - b_i)^2\big) \leq {\color{black} \varphi}\big((a_i^T y - b_i)^2\big) + \omega_i^y\left[(a_i^T x - b_i)^2 -(a_i^T y - b_i)^2\right],
\]
Summing both sides of the above inequality from $i=1$ to $p$, we obtain that
\begin{equation}\label{ell}
\begin{aligned}
\ell(Ax - b) &\leq \ell(Ay - b) + \sum_{i=1}^p \omega_i^y\left[(a_i^T x - b_i)^2 -(a_i^T y - b_i)^2\right]\\
& = \ell(Ay - b) + \ell^y(Ax - b) - \ell^y(Ay-b).
\end{aligned}
\end{equation}

We are now ready to present our algorithm for solving \eqref{P1} under Assumptions~\ref{gassum} and \ref{assumption4}, which is stated as Algorithm~\ref{SQP_retract2} below. This is again an SQP-type method and is also motivated by the retraction strategy. Indeed, the subproblem \eqref{subproblem4} takes the same form as the subproblem \eqref{subproblem2} in Algorithm~\ref{SQP_retract}, with $\ell(Ax - b)$ in place of $g(x)$. The key difference lies in Step 2a). In Algorithm~\ref{SQP_retract}, we ``retract" by taking convex combinations so that $\widetilde x$ satisfies the convex constraint $\{x\in C:\; g_i(x) \le 0,\ \ i = 1,\ldots,m\}$. However, this approach cannot be directly adapted for \eqref{P1}, because $\{x\in C:\; \ell(Ax - b)\le \sigma\}$ is nonconvex. Fortunately, one can deduce from \eqref{defomega}, \eqref{defsigma}, \eqref{defell} and \eqref{ell} that $\ell(Ax - b) - \sigma \le \ell^y(Ax - b) - \widetilde{\sigma}^y$ (see Lemma~\ref{ell+}(v) below); moreover, the function $\ell^y$ is always convex (see Lemma~\ref{ell+}(ii) below). Thus, in Step 2a) of our Algorithm~\ref{SQP_retract2}, we ``retract" so that $\widetilde x$ satisfies the \emph{convex} constraint $\mathcal{F}_1^{y} := \{x\in C:\; \ell^y(Ax - b)\le \widetilde\sigma^y\}$ for some suitable $y \in \R^n$. Since $\ell^y$ is always nonnegative, for this retraction strategy to make sense, we have to at least show that $\widetilde\sigma^y > 0$ at each iteration. We prove this together with some auxiliary facts concerning $\ell^y$ in the next lemma.
\begin{algorithm}
\caption{ FPA with retraction for \eqref{P1} under  Assumptions~\ref{gassum} and \ref{assumption4}}\label{SQP_retract2}
\begin{algorithmic}
\STATE
\begin{description}
  \item[\bf Step 0.] Choose $x^0\in \mathcal{F}_1$, $0 < \underline{\beta} < \overline{\beta}$, $c > 0$ and $\eta \in (0, 1)$. Choose $\xfeasss$ as in Assumption~\ref{assumption4} and set $k = 0$.
  \item[\bf Step 1.] Pick any $\xi^k\in \partial P_2(x^k)$. Compute $\omega_i^{x^k}$ as in \eqref{defomega} and $\tilde{\sigma}^{x^k}$ as in \eqref{defsigma}. Choose $\beta_k^0\in[\underline{\beta}, \overline{\beta}]$ arbitrarily and set $\widetilde{\beta} = \beta_k^0$.
  \item[\bf Step 2.] Compute
  \begin{equation}\label{subproblem4}
  \begin{array}{rl}
  \widetilde{u}= \argmin\limits_{x\in\R^n} & P_1(x) - \langle \xi^k, x - x^k\rangle + (2\widetilde{\beta})^{-1}\|x - x^k\|^2\\ [2pt]
      {\rm s.t.}& \ell(Ax^k - b) + \langle A^T\nabla\ell(Ax^k - b), x-x^k\rangle \le \sigma,\\ [2pt]
      & x\in C.
  \end{array}
  \end{equation}
  \begin{enumerate}[\bf {Step 2}a)]
    \item If $\ell^{x^k}(A\widetilde{u} - b) \le \tilde{\sigma}^{x^k}$, set $\widetilde{x}=\widetilde{u}$ and $\widetilde\tau = 0$; else, find $\widetilde{\tau}\in(0,1)$ such that $\widetilde{x}:=(1 - \widetilde{\tau})\widetilde{u} + \widetilde{\tau} \xfeasss$ satisfies $\ell^{x^k}(A\widetilde{x} - b) = \tilde{\sigma}^{x^k}$.
    \item If $P(\widetilde{x}) \leq P(x^k) - \frac{c}{2}\|\widetilde{u} - x^k\|^2$, go to \textbf{Step 3}; else, update $\widetilde{\beta}\leftarrow \eta\widetilde{\beta}$ and go to \textbf{Step 2}.
  \end{enumerate}
  \item[\bf Step 3.] Set $x^{k+1} = \widetilde{x}$, $u^k = \widetilde{u}$, $\beta_k = \widetilde{\beta}$, $\tau_k = \widetilde{\tau}$. Update $k \leftarrow k+1$ and go to \textbf{Step 1}.
\end{description}
\end{algorithmic}
\end{algorithm}

\begin{lemma}\label{ell+}
 Consider \eqref{P1} and let ${\cal F}_1$ denote its feasible set. Assume that Assumption~\ref{gassum} holds. Let $y\in \R^n$, and let $\tilde \sigma^y$ and $\ell^y$ be defined in \eqref{defomega} and \eqref{defell} respectively. Then the following statements hold:
\begin{enumerate}[{\rm (i)}]
  \item It holds that $\sigmamin:= \inf\limits_{v\in \mathcal{F}_1}\tilde{\sigma}^v > 0$.
  \item $\ell^y(0) = 0$ and $\ell^y$ is convex.
  \item $\ell^y$ is smooth and there exists $L_\ell > 0$ (independent of $y$)\footnote{Indeed, one can take $L_\ell = 2\vartheta$ with $\vartheta$ as in Assumption~\ref{gassum}.} such that for all $u$, $v\in \R^p$, we have
  \begin{equation}\label{Lipineq}
  \|\nabla \ell^y(u) - \nabla \ell^y(v)\| \le L_\ell\|u - v\|.
  \end{equation}
  \item $\ell(Ay - b) - \sigma = \ell^y(Ay - b) - \tilde{\sigma}^y$, and $\nabla\ell(Ay - b) = \nabla\ell^y(Ay - b)$.
  \item $\ell(Ax - b) - \sigma \leq \ell^y(Ax - b) - \tilde{\sigma}^y$ for any $x\in\R^n$.
\end{enumerate}
\end{lemma}
\begin{proof}
(i): Let $v \in \mathcal{F}_1$. Since ${\color{black} \varphi}:\R_+\to \R_+$ is a continuous concave function with ${\color{black} \varphi}(0) = 0$ that is continuously differentiable with positive derivative on $(0, \infty)$, we have $\omega_i^v > 0$ for all $i \notin I_0$, where $I_0=\{i: a_i^Tv-b_i=0\}$. Next, note that $\ell(Av - b) \leq \sigma$. In addition, if $v$ satisfies $\ell(Av - b) = \sigma$, then there must exist $i_0$ such that $a_{i_0}^Tv - b_{i_0} \not= 0$ (since $\ell(0) = 0 < \sigma$), hence
\[
\tilde{\sigma}^v = \sigma - \ell(Av - b) + \sum_{ i \notin I_{0}} \omega_i^v(a_i^Tv - b_i)^2 > 0.
 \]
Finally, from \eqref{defsigma}, we have $z\mapsto \tilde{\sigma}^z$ is continuous. This together with the compactness of $C$ shows that $\sigmamin:= \min\limits_{v\in \mathcal{F}_1}\tilde{\sigma}^v > 0$.

(ii): From \eqref{defell}, one can see that $\ell^y(0) = 0$. The convexity of $\ell^y$ follows from \eqref{defomega} and Assumption~\ref{gassum}.

(iii): By the definition of $\ell^y$, we have $\nabla \ell^y (u) = (2\omega_i^yu_i)_{i=1}^p$, and $\nabla^2\ell^y (u) = 2\Diag(\omega^y)$.\footnote{Here, $\Diag(\omega^y)$ is the diagonal matrix whose $i$-th diagonal entry is $\omega^y_i$.} Now, according to \cite[Lemma~2.2]{YuPo19}, for any $y$, we have that $0\le\omega_i^y\leq \lim_{t\downarrow 0}{\color{black} \varphi}'(t)=\vartheta$. This implies that \eqref{Lipineq} holds with $L_{\ell} = 2\vartheta$.

(iv): The first equation follows from \eqref{defsigma} and the definition of $\ell^y$. As for the second equation, note that for any $y\in\R^n$, using the definitions of $\ell$ and $\ell^y$, we have $\nabla\ell(u) = \left(2{\color{black} \varphi}_+'(u_i^2)u_i\right)_{i=1}^p$ and $\nabla\ell^y(u) = \left(2{\color{black} \varphi}_+'((a_i^Ty - b_i)^2)u_i\right)_{i=1}^p$. Therefore,
\[
 \nabla\ell(Ay - b) = \left(2{\color{black} \varphi}_+'((a_i^Ty - b_i)^2)(a_i^Ty - b_i)\right)_{i=1}^p = \nabla\ell^y(Ay - b).
\]

(v): From \eqref{ell}, we see that $\ell(Ax - b) - \ell^y(Ax - b) \leq \ell(Ay - b) - \ell^y(Ay - b)$ for any $x\in \R^n$. Combining this with item (iv), we deduce that $\ell(Ax - b) - \sigma \leq \ell^y(Ax - b) - \tilde{\sigma}^y$ for any $x\in \R^n$. This completes the proof.
\end{proof}

Now, we prove the well-definedness of Algorithm~\ref{SQP_retract2}. Recall that ${\cal F}_1$ denotes the feasible set of \eqref{P1}. To prove well-definedness, it suffices to show that, if an $x^k\in {\cal F}_1$  for some $k\ge 0$, then the subproblem~\eqref{subproblem4} has a unique solution, $\widetilde\tau$ exists in Step 2a) and the line-search step in Step 2b) terminates in finitely many inner iterations so that $x^{k+1}$ can be generated: Note that such an $x^{k+1}$ belongs to ${\cal F}_1$ in view of Step 2a) and Lemma~\ref{ell+}(v). This together with an induction argument would establish well-definedness.

\begin{theorem}[{{\bf Well-definedness of Algorithm~\ref{SQP_retract2}}}]\label{welldef2}
Consider \eqref{P1} and suppose that Assumptions~\ref{gassum} and \ref{assumption4} hold. Suppose that an $x^k\in {\cal F}_1$ is generated {\color{black} at the beginning of the $k$-th iteration} of Algorithm~\ref{SQP_retract2} for some $k\ge 0$. Then the following statements hold:
\begin{enumerate}[{\rm (i)}]
   \item The convex subproblem \eqref{subproblem4} has a unique solution. Moreover, $\xfeasss$ from Assumption~\ref{assumption4} is a Slater point of the feasible set of \eqref{subproblem4}.
   \item If $\ell^{x^k}(A\widetilde{u} - b) > \tilde{\sigma}^{x^k}$, then there exists $\widetilde{\tau}\in (0, 1)$ such that $\widetilde{x}:=(1 - \widetilde{\tau})\widetilde{u} + \widetilde{\tau} \xfeasss$ with $\ell^{x^k}(A\widetilde{x} - b) = \tilde{\sigma}^{x^k}$. Moreover, for any such $\widetilde\tau$, it holds that
       \[
       \widetilde{\tau} \leq \frac{L_{\ell}\|A\|^2}{2\sigmamin}\|\widetilde{u} - x^k\|^2
       \ \ {\it and}\ \ \|\widetilde{x} - \widetilde{u}\| \leq \frac{ML_{\ell}\|A\|^2}{2\sigmamin}\|\widetilde{u} - x^k\|^2,
       \]
       where $M := \sup\limits_{x\in C}\|x - \xfeasss\|$, and $\sigmamin$ and $L_\ell$ are defined as in Lemma~\ref{ell+}.
   \item Step 2 of algorithm \ref{SQP_retract2} terminates in finitely many inner iterations.
   \item For each $\widetilde{\beta} > 0$, the subproblem \eqref{subproblem4} has a Lagrange multiplier $\widetilde\lambda\geq 0$ and it satisfies
       \begin{equation*}
       0\in \partial P_1(\widetilde{u}) - \xi^k +\frac{1}{\widetilde{\beta}}(\widetilde{u} - x^k) +\widetilde\lambda A^T\nabla\ell(Ax^k - b) + \mathcal{N}_C(\widetilde{u}),
       \end{equation*}
       and
       \begin{equation*}
       \widetilde\lambda\big(\ell(Ax^k - b) + \langle A^T\nabla\ell(Ax^k - b), \widetilde{u}-x^k\rangle - \sigma\big) = 0.
       \end{equation*}
 \end{enumerate}
 In particular, an $x^{k+1}\in {\cal F}_1$ can be generated at the end of the $(k + 1)$-th iteration of Algorithm~\ref{SQP_retract2}.
\end{theorem}
\begin{proof}
(i): We first observe that the feasible region of \eqref{subproblem4} is  closed and convex. We next show that it is nonempty by showing that $\xfeasss$ is a Slater point of the feasible set of \eqref{subproblem4}.
To this end, we note from the convexity of $\ell^{x^k}$ (see Lemma~\ref{ell+}(ii)) that
\begin{equation}\label{relxxx}
\ell^{x^k}(Ax^k - b) + \langle A^T\nabla\ell^{x^k}(Ax^k - b), \xfeasss - x^k\rangle\leq\ell^{x^k}(A\xfeasss - b) \overset{\rm (a)}= \ell^{x^k}(0) \overset{\rm (b)}= 0 \overset{\rm (c)}< \tilde{\sigma}^{x^k},
\end{equation}
where (a) follows from Assumption~\ref{assumption4}, (b) follows from Lemma~\ref{ell+}(ii), and (c) holds because of Lemma~\ref{ell+}(i) and the fact that $x^k\in {\cal F}_1$. Hence,
\[
\begin{aligned}
  &\ell(Ax^k - b) + \langle A^T\nabla\ell(Ax^k - b), \xfeasss - x^k\rangle\\
  & = \sigma + \ell^{x^k}(Ax^k - b)- \tilde{\sigma}^{x^k} + \langle A^T\nabla\ell^{x^k}(Ax^k - b), \xfeasss - x^k\rangle < \sigma,
\end{aligned}
\]
where the equality follows from Lemma~\ref{ell+}(iv) and the inequality holds thanks to \eqref{relxxx}.
Since we also have $\xfeasss\in C$ by assumption, we have shown that $\xfeasss$ is a Slater point for \eqref{subproblem4}. Next, since $P_1$ is convex and continuous, it holds that $x\mapsto P_1(x) - \langle \xi^k, x - x^k\rangle + (2\widetilde{\beta})^{-1}\|x - x^k\|^2$ is strongly convex and continuous, for any $\widetilde{\beta} > 0$. Therefore, \eqref{subproblem4} has a unique solution.

(ii): Suppose that $\ell^{x^k}(A\widetilde{u} - b) > \tilde{\sigma}^{x^k}$. By Assumption \ref{assumption4}, Lemma~\ref{ell+}(i) and (ii), and the fact that $x^k\in\mathcal{F}_1$, we have
\begin{equation}\label{circstrictineq}
\ell^{x^k}(A\xfeasss - b) - \tilde{\sigma}^{x^k} = -\tilde{\sigma}^{x^k} < 0.
\end{equation}
Letting $x_{\tau}:=(1 - \tau)\widetilde{u} + \tau \xfeasss$, then $\ell^{x^k}(Ax_{1} - b) - \tilde{\sigma}^{x^k} <0$ and $\ell^{x^k}(Ax_{0} - b) - \tilde{\sigma}^{x^k} >0$. Hence, there exists $\widetilde{\tau}\in (0, 1)$ such that $\widetilde{x} := (1 - \widetilde{\tau})\widetilde{u} + \widetilde{\tau} \xfeasss$ satisfies $\ell^{x^k}(A\widetilde{x} - b) - \tilde{\sigma}^{x^k} = 0$.

Now, recall from Lemma~\ref{ell+}(ii) that $\ell^{x^k}$ is convex and thus
\[
0 = \ell^{x^k}(A\widetilde{x} - b) - \tilde{\sigma}^{x^k} \leq (1 - \widetilde{\tau})\left(\ell^{x^k}(A\widetilde{u} - b) - \tilde{\sigma}^{x^k}\right) + \widetilde{\tau}\left(\ell^{x^k}(A\xfeasss - b) - \tilde{\sigma}^{x^k}\right).
\]
Upon rearranging terms in the above inequality, we see that
\begin{equation}\label{taubound}
\begin{aligned}
\widetilde{\tau}&\overset{\rm (a)}\leq\frac{\ell^{x^k}(A\widetilde{u} - b) - \tilde{\sigma}^{x^k}}{\left(\ell^{x^k}(A\widetilde{u} - b) - \tilde{\sigma}^{x^k}\right) + \tilde{\sigma}^{x^k}} \overset{\rm (b)}\leq \frac{\ell^{x^k}(A\widetilde{u} - b) - \tilde{\sigma}^{x^k}}{\tilde{\sigma}^{x^k}} \\
&\overset{\rm (c)}\le \frac{\ell^{x^k}(Ax^k - b) + \langle A^T\nabla \ell^{x^k}(Ax^k - b),\widetilde{u}-x^k\rangle - \tilde{\sigma}^{x^k} + (L_{\ell}\|A\|^2/2)\|\widetilde{u} - x^k\|^2}{\tilde{\sigma}^{x^k}}\\
&\overset{\rm (d)}\leq \frac{L_{\ell}\|A\|^2}{2\sigmamin}\|\widetilde{u} - x^k\|^2,
\end{aligned}
\end{equation}
where (a) and (b) hold thanks to $\ell^{x^k}(A\widetilde{u} - b) - \tilde{\sigma}^{x^k} > 0$ and \eqref{circstrictineq}, (c) holds because $\ell^{x^k}$ is smooth with $L_\ell$-Lipschitz gradient (see Lemma~\ref{ell+}(iii)), and (d) holds in view of Lemma~\ref{ell+}(iv) and the feasibility of $\widetilde u$ for \eqref{subproblem4}, and the fact that $\tilde{\sigma}^{x^k} \ge \sigmamin > 0$ (thanks to Lemma \ref{ell+}(i) and $x^k\in\mathcal{F}_1$). Hence,
\[
\begin{aligned}
\|\widetilde{x} - \widetilde{u}\|& = \widetilde{\tau}\|\widetilde{u} - \xfeasss\|\overset{\rm (a)}\leq \widetilde{\tau}M \overset{\rm (b)}\leq \frac{ML_{\ell}\|A\|^2}{2\sigmamin}\|\widetilde{u} - x^k\|^2,
\end{aligned}
\]
where $M := \sup\limits_{x\in C}\|x - \xfeasss\|$ and $\rm (a)$ holds because $\widetilde{u}\in C$, and $\rm(b)$ follows from \eqref{taubound}.

(iii): Note that $\widetilde{x} = (1 - \widetilde{\tau})\widetilde{u} + \widetilde{\tau} \xfeasss$ for some $\widetilde{\tau}\in [0,1]$ and thus
\[
\begin{aligned}
P(\widetilde{x}) &= P_1(\widetilde{x}) - P_2(\widetilde{x}) \overset{\rm (a)}\leq P_1(\widetilde{x}) - P_2(x^k) - \langle \xi^k, \widetilde{x} - x^k\rangle \\
&\overset{\rm (b)}\leq (1 - \widetilde{\tau})P_1(\widetilde{u}) + \widetilde{\tau}P_1(\xfeasss) - P_2(x^k) - (1 - \widetilde{\tau})\langle \xi^k, \widetilde{u} - x^k\rangle - \widetilde{\tau}\langle \xi^k, \xfeasss - x^k\rangle\\
&= P_1(\widetilde{u}) - \langle \xi^k, \widetilde{u} - x^k\rangle + \widetilde{\tau}\big(P_1(\xfeasss) - P_1(\widetilde{u}) - \langle \xi^k, \xfeasss - \widetilde{u}\rangle \big) - P_2(x^k)\\
&\overset{\rm (c)}\leq P_1(x^k) - \frac{1}{2\widetilde{\beta}}\|\widetilde{u} - x^k\|^2  - P_2(x^k) + \widetilde{\tau}\big(P_1(\xfeasss) - P_1(\widetilde{u}) - \langle \xi^k, \xfeasss - \widetilde{u}\rangle \big)\\
&\overset{\rm (d)}\leq P(x^k) - \frac{1}{2\widetilde{\beta}}\|\widetilde{u} - x^k\|^2 + \widetilde{\tau}M_1 \overset{\rm (e)}\leq P(x^k) - \left(\frac{1}{2\widetilde{\beta}} - \frac{M_1L_{\ell}\|A\|^2}{2\sigmamin}\right)\|\widetilde{u} - x^k\|^2,
\end{aligned}
\]
where (a) follows from the convexity of $P_2$, (b) holds because $P_1$ is convex, (c) follows from the facts that $\widetilde{u}$ is the minimizer of \eqref{subproblem4} and $x^k$ is feasible for \eqref{subproblem4} (since $x^k\in {\cal F}_1$), (d) holds with $M_1:= \sup\limits_{y,z\in C}\sup\limits_{\xi\in\partial P_2(z)}\{P_1(\xfeasss) - P_1(y) - \langle\xi, \xfeasss - y\rangle\}\in [0,\infty)$ (note that $M_1 < \infty$ thanks to \cite[Theorem~2.6]{Tu98} and the compactness of $C$, and $M_1\ge 0$ since $\xfeasss\in C$), (e) follows from item~(ii).

Therefore, as long as $\widetilde{\beta} < \frac{1}{2}\left(\frac{c}2 + \frac{M_1L_{\ell}\|A\|^2}{2\sigmamin}\right)^{-1}$ (independent of $k$), we have $P(\widetilde{x}) \leq P(x^k) - \frac{c}{2}\|\widetilde{u} - x^k\|^2$. In view of {Step 2b)} of Algorithm \ref{SQP_retract}, we conclude that the inner loop terminates finitely.

(iv): Since \eqref{subproblem4} is convex and the Slater condition holds by item (i), the desired conclusion follows from \cite[Corollary~28.2.1]{Ro70} and \cite[Theorem~28.3]{Ro70}.
\end{proof}

\begin{remark}[{{\bf Uniform lower bound on $\beta_k$}}]\label{gamma2}
From the proof of Theorem~\ref{welldef2}(iii), we see that for the $\beta_k$ in Algorithm~\ref{SQP_retract2}, it holds that
\[
\inf_k \beta_k \ge \widehat{\beta}_{\rm min}:= \textstyle\min\left\{\frac{\eta}{2}\left(\frac{c}2 + {\frac{M_1L_{\ell}\|A\|^2}{2\sigmamin}} \right)^{-1},\underline{\beta}\right\},
\]
where $\eta$, $c$ and $\underline{\beta}$ are given in Step 0 of Algorithm~\ref{SQP_retract2}, $\sigmamin$ and $L_\ell$ are defined as in Lemma~\ref{ell+}, and $M_1:= \sup\limits_{y,z\in C}\sup\limits_{\xi\in\partial P_2(z)}\{P_1(\xfeasss) - P_1(y) - \langle\xi, \xfeasss - y\rangle\}$.
\end{remark}

\begin{remark}[{{\bf Finding $\widetilde\tau$}}]
In Step 2a) of Algorithm~\ref{SQP_retract2}, when $\ell^{x^k}(A\widetilde u - b) > \tilde{\sigma}^{x^k}$, we need to find $\widetilde\tau$ so that $\ell^{x^k}(A((1-\widetilde\tau)\widetilde u + \widetilde\tau \xfeasss) - b) = \tilde{\sigma}^{x^k}$. We note that such a parameter $\widetilde\tau$ exists thanks to Theorem~\ref{welldef2}(ii), and admits a closed form formula since $\ell^{x^k}$ is a quadratic function.
\end{remark}

Before presenting the convergence result of Algorithm~\ref{SQP_retract2}, we show that the MFCQ holds for \eqref{P1}, under Assumptions~\ref{gassum} and \ref{assumption4}.
\begin{proposition}\label{MFCQ}
Consider \eqref{P1} and suppose that Assumptions~\ref{gassum} and \ref{assumption4} hold. Then the MFCQ for \eqref{P1} holds at any $y\in \mathcal{F}_1$, i.e., for any $y\in \mathcal{F}_1$, the following implication holds:
\[
\left.\begin{matrix}
-\lambda A^T\nabla\ell(Ay - b)\in \mathcal{N}_C(y)\\
\lambda(\ell(Ay - b) -\sigma) =0\\
\lambda\geq 0
\end{matrix}\right\} \Rightarrow\lambda =0.
\]
\end{proposition}
\begin{proof}
Suppose that $-\lambda A^T\nabla\ell(Ay - b)\in \mathcal{N}_C(y)$ and $\lambda(\ell(Ay - b) -\sigma) =0$ but $\lambda > 0$. We will derive a contradiction.

To this end, recall from Lemma~\ref{ell+}(iv) that $\ell(Ay - b) - \sigma = \ell^y(Ay - b) - \tilde{\sigma}^y$ and $\nabla\ell(Ay - b) = \nabla\ell^y(Ay - b)$. Thus, we have
 \begin{equation}\label{hahahahaha}
 -\lambda A^T\nabla\ell^y(Ay - b)\in \mathcal{N}_C(y)\  \ \text{ and }\  \ \lambda(\ell^y(Ay - b) - \tilde{\sigma}^y) =0.
 \end{equation}
Using the second relation above, the convexity of $\ell^y$ (see Lemma~\ref{ell+}(ii)) and the assumption that $\lambda > 0$, we have
 \[
 \begin{aligned}
\langle\lambda A^T\nabla\ell^y(Ay - b), \xfeasss - y\rangle &\leq \lambda\ell^y(A\xfeasss - b) - \lambda\ell^y(Ay - b)  = \lambda\ell^y(A\xfeasss - b)-  \lambda\tilde{\sigma}^y = -  \lambda\tilde{\sigma}^y < 0,
 \end{aligned}
 \]
where the second equality follows from Assumption~\ref{assumption4} and Lemma~\ref{ell+}(ii), and the strict inequality holds in view of $\lambda > 0$, $y\in {\cal F}_1$ and Lemma~\ref{ell+}(i).
On the other hand, the first relation in \eqref{hahahahaha} shows that $\langle-\lambda A^T\nabla\ell^y(Ay - b), x - y\rangle\leq 0$ for any $x\in C$, which contradicts the above display because $\xfeasss\in C$. Thus, we must have $\lambda = 0$. This completes the proof.
\end{proof}
We now present the main convergence result of Algorithm~\ref{SQP_retract2}. The proof is almost identical to that of Theorem~\ref{convergence2}, except that Theorem~\ref{welldef2}(iv) is used in place of Remark~\ref{KKT}, and Remark~\ref{gamma2} is used in place of Remark~\ref{gamma}. We omit its proof for brevity.
\begin{theorem}[{{\bf Subsequential convergence of Algorithm~\ref{SQP_retract2}}}]\label{convergence3}
Consider \eqref{P1} and suppose that Assumptions~\ref{gassum} and \ref{assumption4} hold. Let {\color{black}$\{(x^k,u^k)\}$} be the sequence generated by Algorithm \ref{SQP_retract2} and $\lambda_k$ be a Lagrange multiplier of \eqref{subproblem4} with $\widetilde{\beta} = \beta_k$.  Then the following statements hold:
\begin{enumerate}[{\rm(i)}]
  \item It holds that $\lim\limits_{k\rightarrow\infty}\|u^k - x^k\| = 0$.
  \item The sequences $\{x_k\}$ and $\{\lambda_k\}$ are bounded.
  \item Any accumulation point of $\{x^k\}$ is a {\color{black}critical point} of \eqref{P1}.
\end{enumerate}
\end{theorem}

\section{Explicit Kurdyka-{\L}ojasiewicz exponent for some structured convex optimization models}\label{sec5}
In this section, we examine the KL exponent for some structured convex optimization model that arises from applications. Specifically, we consider the following constrained optimization problem:
\begin{equation}\label{KLproblem}
\min\limits_{x\in \R^n}F(x):= h(x) + \delta_{G(\cdot) \leq 0}(x),
\end{equation}
where $h$ is proper, closed and convex, the function $G(x) = (q_1(A_1x), \ldots, q_m(A_m(x)))$ is continuous with each $q_i: \R^{s_i}\to \R$ being strictly convex and each $A_i\in\R^{s_i\times n}$, and $\{x\in {\rm dom}\,h: G(x)\le 0\}\not=\emptyset$. One can see that, when $q_i$'s are additionally Lipschitz differentiable and $h = P_1 + \delta_C$ for some continuous convex function $P_1$ and nonempty compact convex set $C$, problem \eqref{KLproblem} is a special case of \eqref{P0} with $P_2 = 0$ and $g_i = q_i\circ A_i$ for $i = 1, \ldots, m$.

The proof of the following theorem is similar to that of \cite[Theorem~4.2]{YuLP20}. We omit its proof for brevity.
\begin{theorem}\label{KL}
Let $F$ be as in \eqref{KLproblem} and $\bar{x}\in\Argmin F$. Suppose the following conditions hold:
\begin{enumerate}[{\rm (i)}]
  \item There exists a Lagrange multiplier $\bar{\lambda}\in\R^m_+$ for \eqref{KLproblem} and $x\mapsto h(x) + \langle\bar{\lambda}, G(x)\rangle$ is a KL function with exponent $\alpha\in(0,1)$.
  \item The strict complementarity condition holds at $(\bar{x}, \bar{\lambda})$, i.e., for any $i$ satisfying $\bar{\lambda}_i = 0$, it holds that $q_i(A_i\bar{x}) < 0$.
 \end{enumerate}
 Then $F$ satisfies the KL property with exponent $\alpha$ at $\bar{x}$.
\end{theorem}

The next corollary deals with \eqref{KLproblem} with $m = 1$. Its proof follows a similar line of arguments as the proof of \cite[Corollary~4.3]{YuLP20}. We omit the proof for brevity.

\begin{corollary}\label{KL1}
Let $F$ be as in \eqref{KLproblem} with $m = 1$. Suppose the following conditions hold:
\begin{enumerate}[{\rm (i)}]
   \item $\inf\limits_{x\in \R^n}h(x) < \inf\limits_{x\in \R^n} \{h(x):\; q_1(A_1x)\leq 0\}$.
   \item There exists a Lagrange multiplier $\bar{\lambda}\in\R_+$ for \eqref{KLproblem} and $x\mapsto h(x) + \bar{\lambda}q_1(A_1x)$ is a KL function with exponent $\alpha\in(0,1)$.
 \end{enumerate}
 Then $F$ is KL function with exponent $\alpha$.
\end{corollary}
Theorem~\ref{KL} and Corollary~\ref{KL1} give sufficient conditions for us to derive the KL exponent of $F$ in \eqref{KLproblem} from that of its Lagrangian. In the following subsections, we will make use of Corollary~\ref{KL1} to derive the KL exponent of \eqref{KLproblem} with $m = 1$ and specific choices of $h$. We will need the following assumption, which is a commonly used sufficient condition for the existence of Lagrange multipliers.
\begin{assumption}\label{assumptionxxx}
  For \eqref{KLproblem} with $m = 1$, there exists $\xfeas\in \dom h$ with $q_1(A_1\xfeas) < 0$.
\end{assumption}

\subsection{Case 1: $h$ is the sum of the $\ell_1$ norm and the indicator function of a polyhedron containing the origin}
Suppose that $h = \|\cdot\|_1 + \delta_C$ in \eqref{KLproblem}, where $C$ is polyhedron containing the origin. We deduce from \cite[Corollary~5.1]{LiPong18} and Corollary~\ref{KL} that the KL exponent of the corresponding $F$ in \eqref{KLproblem} is $\frac{1}{2}$ if $m = 1$, $q_1:\mathbb{R}^{s_1} \rightarrow \mathbb{R}$ takes one of the following forms with $b\in \R^{s_1}$ and $\sigma > 0$ chosen so that the origin is not feasible and that Assumption~\ref{assumptionxxx} holds:
\begin{enumerate}[{\rm (i)}]
   \item (Basis pursuit denoising \cite{Ca18}) $q_1(z) = \frac{1}{2}\|z - b\|^2 - \sigma$.
   \item (Logistic loss \cite{HoLS13}) $q_1(z) = \sum_{i=1}^{s_1}\log(1 + \exp(b_iz_i)) - \sigma$ for some $b\in \{-1,1\}^{s_1}$.
   \item (Poisson loss \cite{zo04}) $q_1(z) = \sum_{i=1}^{s_1}(-b_iz_i + \exp(z_i)) - \sigma$.
 \end{enumerate}

\subsection{Case 2: $h$ is the sum of a closed gauge and the Euclidean norm}
When $h = \|\cdot\| + \kappa$, where $\kappa$ is a closed gauge, i.e., a nonnegative, positively homogeneous closed convex function that vanishes at the origin, the corresponding problem \eqref{KLproblem} with $m = 1$ and a suitable choice of $q_1$ was studied in \cite[Section~6.1]{FrMP19} as an approach to ``regularize" an underlying gauge optimization problem. In what follows, we consider the following assumption on $q_1$ in \eqref{KLproblem} with $m = 1$.
\begin{assumption}\label{EllStrCon}
For \eqref{KLproblem} with $m = 1$, the function $q_1$ is strongly convex on any compact convex set and is twice continuously differentiable. Moreover, $\inf q_1 > -\infty$.
\end{assumption}
Conditions on $q_1$ similar to those in Assumption~\ref{EllStrCon} can be found in \cite{ZhSo17,LiPong18} and are standard assumptions when it comes to studying the KL property of functions of the form $\lambda \, q_1(A_1x) + h(x)$, where $\lambda > 0$; these are Lagrangian functions of the $F$ in \eqref{KLproblem} with $m = 1$. In the next theorem, we study the KL property of \eqref{KLproblem} with $m = 1$ under Assumptions~\ref{assumptionxxx} and~\ref{EllStrCon}. Our analysis makes use of Corollary~\ref{KL1}, which means that we will first study the KL property of some suitable Lagrangian functions. The proof of item (ii) below follows closely the arguments in \cite[Theorem~4.4]{YuLiPo19} concerning the so-called $C^2$-cone-reducible sets: Indeed, our argument is based on representing a certain set as the level set of a continuously differentiable function with Lipschitz continuous gradient.

\begin{theorem}\label{KLnorm}
Let $h = \|\cdot\| + \kappa$ for some closed gauge $\kappa$. Then the following statements hold:
\begin{enumerate}[{\rm (i)}]
   \item For any $\zeta\in \R^n$, the set $(\partial h)^{-1}(\zeta)$ is a polyhedron.
   \item The mapping $x\mapsto \partial h(x)$ is metrically subregular at $x\in\dom \partial h$ for any $\zeta\in \partial h(x)$.
   \item Consider \eqref{KLproblem} with $m = 1$ and $h$ as above. If Assumptions~\ref{assumptionxxx} and \ref{EllStrCon} hold and the origin is not feasible, then the function $F$ in \eqref{KLproblem} is a KL function with an exponent of $\frac{1}{2}$.
 \end{enumerate}
\end{theorem}
\begin{proof}
(i): Since $\kappa$ is a closed gauge, from \cite[Proposition~2.1(iii)]{FrMP14}, we see that $\kappa(x) = \sigma_{D}(x)$, where $D := \{x\in\R^n:\;\kappa^\circ(x) \leq 1\}$ is a convex closed set, and $\kappa^\circ$ is the polar of the gauge $\kappa$. Combining this with $\|x\| = \sigma_{B}(x)$, where $B:= B(0,1)$, we have that
\begin{equation}\label{inver}
(\partial h)^{-1} =(\partial(\|\cdot\| + \kappa))^{-1} = (\partial\sigma_{B + D})^{-1} \overset{\rm (a)}= \mathcal{N}_{B + D},
\end{equation}
where (a) follows from \cite[Eq.~11(4)]{RoWe98}.

Define $\Theta: \R^{n}\to \R$ by $\Theta(v) := \d^2(v,D) - 1$. Then we have $\nabla\Theta(v) = 2(v - P_D(v))$, where $P_D(v)$ denotes the projection of $v$ onto $D$, and
\[
B + D = \left\{v :\; \d^2(v,D) - 1 \leq 0\right\} = \{v :\; \Theta(v) \le 0 \}.
\]
Moreover, $\Theta(v) = 0 \Leftrightarrow \|v - P_D(v)\| = 1$. In addition, from \cite[Section~D:~Theorem~1.3.5]{HirLem01}, we have for any $v\in \R^{n}$ that
\begin{equation}\label{Nor}
\mathcal{N}_{B + D}(v) = \begin{cases}
  \{0\} & {\rm if}\ \Theta(v) < 0,\\
  \bigcup_{t \ge 0}t(v - P_D(v)) & {\rm if}\ \Theta(v) = 0,\\
  \emptyset & {\rm otherwise}.
\end{cases}
\end{equation}
Using \eqref{Nor} together with \eqref{inver}, we conclude that $(\partial h)^{-1}(\zeta)$ is a polyhedron for any $\zeta\in\R^n$, which completes the proof of (i).

(ii): Fix any $x\in\dom \partial h$ and any $\zeta\in \partial h(x)$. In order to prove that the mapping $x\mapsto \partial h(x)$ is metrically subregular at $x$ for $\zeta$, by \cite[Theorem~3H.3]{DonRo09}, we only need to prove that the set-valued mapping $(\partial h)^{-1}$ is calm at $\zeta$ with respect to $x$. By \eqref{inver}, it suffices to show that the mapping $v\mapsto \mathcal{N}_{B + D}(v)$ is calm at $\zeta$ with respect to $x$, i.e., there exist $\delta,\rho,\eta > 0$, such that
whenever $v\in B(\zeta,\rho)$,
\begin{equation}\label{calm}
\mathcal{N}_{B + D}(v)\cap B(x, \delta)\subseteq \mathcal{N}_{B + D}(\zeta) + \eta\|v - \zeta\|B(0,1).
\end{equation}

To this end, we note first from \eqref{inver} and \eqref{Nor} that $x\in{\cal N}_{B+D}(\zeta)$, which implies in particular that $\zeta \in B+D$. Hence, $\Theta(\zeta)\leq 0$.
Now, if $\Theta(\zeta) < 0$, in view of \eqref{Nor}, we have $\mathcal{N}_{B + D}(\zeta) = \{0\}$. In addition, by the continuity of $\Theta$, there exists $\rho_1 > 0$ such that $\Theta(v) < 0$ whenever $v\in B(\zeta, \rho_1)$, which means that $\mathcal{N}_{B + D}(v) = \{0\}$. Thus, \eqref{calm} holds for $\rho = \rho_1 > 0$ and any $\delta>0$ and $\eta>0$ in this case.

Next, suppose that $\Theta(\zeta) = 0$. Then $\|\zeta - P_D(\zeta)\| = 1 > 0$. Set
\begin{equation}\label{defkappa}
\delta = 1,\ \ \rho = 0.25\ \ {\rm and}\ \ \eta = 4(1 + \|x\|),
\end{equation}
and fix any $v\in B(\zeta, \rho)$. If $\Theta(v)>0$, then $\mathcal{N}_{B + D}(v) = \emptyset$ in view of \eqref{Nor}. On the other hand, if $\Theta(v)<0$, then $\mathcal{N}_{B + D}(v) = \{0\}$. Therefore, when $\Theta(v) \not= 0$, one can check that \eqref{calm} holds with $\delta$ and $\eta$ as in \eqref{defkappa}.

Now, suppose that $\Theta(v) = 0$. Then we have from \eqref{Nor} that $\mathcal{N}_{B + D}(v) = \R_+(v - P_D(v))$. Thus, for any $x_1\in \mathcal{N}_{B + D}(v)\cap B(x, \delta)$, there exists $t\geq 0$ such that $x_1 = t(v - P_D(v))$ and $\|x_1 - x\|\leq \delta = 1$. Then
\[
\begin{aligned}
1&\geq\|x_1 - x\| = \|t(v - P_D(v)) - x\|\geq t\|v - P_D(v)\| - \|x\|\\
& \ge t\|\zeta - P_D(\zeta)\| - t\|[v - P_D(v)] - [\zeta - P_D(\zeta)]\| - \|x\|\\
& \overset{\rm (a)}\ge t\|\zeta - P_D(\zeta)\| - 2t\|v - \zeta\| - \|x\|\overset{\rm (b)}\ge t\|\zeta - P_D(\zeta)\| - 2t\rho - \|x\| = 0.5 t - \|x\|,
\end{aligned}
\]
where (a) follows from the Lipschitz continuity of $P_D(\cdot)$, as $D$ is a closed convex set, (b) holds because $v \in B(\zeta, \rho)$, and the last equality holds because $\|\zeta - P_D(\zeta)\|= 1$ and $\rho = 0.25$. Rearranging terms in the above inequality, we have $t \leq 2(1 + \|x\|)$. Hence,
\[
\begin{aligned}
x_1 = t(v - P_D(v)) &= t(\zeta - P_D(\zeta)) + t(v - P_D(v) - (\zeta - P_D(\zeta)))\\
&\overset{\rm (a)}\in \mathcal{N}_{B + D}(\zeta) + t\|v - P_D(v) - (\zeta - P_D(\zeta))\|B(0,1)\\
&\overset{\rm (b)}\subseteq \mathcal{N}_{B + D}(\zeta) + 2t\|v - \zeta\|B(0,1)\overset{\rm (c)}\subseteq \mathcal{N}_{B + D}(\zeta) + \eta\|v - \zeta\|B(0,1),
\end{aligned}
\]
where (a) follows from \eqref{Nor}, (b) holds because of the Lipschitz continuity of $P_D(\cdot)$, and (c) follows from the fact that $t\le 2(1 + \|x\|)$ and the definition of $\eta$ in \eqref{defkappa}. Thus, we have shown that the mapping $v\mapsto \mathcal{N}_{B + D}(v)$ is calm at $\zeta$ with respect to $x$. This completes the proof of (ii).

(iii): Notice that $\argmin h = \{0\}$ and the origin is not feasible for \eqref{KLproblem} by assumption. Thus, from Corollary~\ref{KL}, in order to prove that the function $F$ in \eqref{KLproblem} with $m = 1$ is a KL function with exponent $\frac{1}{2}$, we only need to show that $F_{\bar{\lambda}}:= h + \bar{\lambda}\, q_1\circ A_1$ is a KL function with exponent $\frac{1}{2}$ for all $\bar{\lambda} > 0$.\footnote{Indeed, Assumption (i) in Corollary~\ref{KL1} implies that any Lagrange multiplier (which exists thanks to Assumption~\ref{assumptionxxx} and the coerciveness of $h$) has to be positive.} To this end, according to \cite[Theorem~4.1]{LiPong18}, it suffices to show that $F_{\bar{\lambda}}$ satisfies the Luo-Tseng error bound.
Clearly, $\Argmin F_{\bar\lambda}$ is compact and nonempty for any $\bar \lambda > 0$ because $h$ is coercive and $\inf q_1 > -\infty$ by Assumption~\ref{EllStrCon}.

Fix any $\bar\lambda > 0$ and any $\bar{x} \in\Argmin F_{\bar{\lambda}}$. Since $q_1$ is strictly convex and $\bar \lambda > 0$, we have that $A_1x\equiv A_1\bar x$ over $\Argmin F_{\bar{\lambda}}$. Thus, we can derive that $\bar\lambda A_1^T\nabla q_1(A_1x)\equiv \bar\lambda A_1^T\nabla q_1(A_1\bar x)=: \bar{\zeta}$ over $\Argmin F_{\bar{\lambda}}$. Moreover, in view of \cite[Theorem~10.1]{RoWe98} and \cite[Exercise~8.8(c)]{RoWe98}, we have $-\bar{\zeta}\in\partial h(\bar{x})$. Now, as in \cite[Section~3.3]{ZhSo17}, we define
\[
\Gamma_f(y):= \left\{x\in\R^n:\; A_1x = y\right\} \text{  and  } \Gamma_{h}(\zeta):= \left\{x\in\R^n:\; -\zeta\in\partial h(x)\right\}.
\]
Then $ \Gamma_{h}(\bar{\zeta}) = (\partial h)^{-1}(-\bar{\zeta})$ and we see from item (i) that this is a polyhedron. This together with \cite[Corollary~3]{BaBL99} shows that $\{\Gamma_f(A_1\bar{x}), \Gamma_{h}(\bar{\zeta})\}$ is boundedly linearly regular. Next, since $-\bar{\zeta}\in\partial h(\bar{x})$, we have from item (ii) that the subdifferential mapping $\partial h$ is metrically subregular at $\bar{x}$ for $-\bar{\zeta}$. Thus, in view of \cite[Theorem~2]{ZhSo17}, \cite[Corollary~1]{ZhSo17} and the arbitrariness of $\bar{x}\in\Argmin F_{\bar{\lambda}}$, we conclude that $F_{\bar{\lambda}}$ satisfies the Luo-Tseng error bound. This completes the proof.
\end{proof}

\subsection{Case 3: $h$ is the sum of group LASSO penalty and a block-structured closed gauge}
In this section, we suppose that
\begin{equation}\label{sec53h}
  h(x) = \sum\limits_{J\in\mathcal{J}}\big(\|x_J\| + \kappa_J(x_J)\big),
\end{equation}
where $\mathcal{J}$ is a partition of $\{1,2,\cdots, n\}$, $x_J\in\R^{|J|}$ is the subvector of $x\in \R^n$ indexed by $J\in\mathcal{J}$, and $\kappa_J: \R^{|J|} \to \R\cup\{\infty\}$ is a closed gauge for each $J\in\mathcal{J}$. The group LASSO regularizer, which is convex but not polyhedral in general, is motivated by the desire to induce sparsity among variables at a group level and has found applications in signal processing and statistics; see, e.g., \cite{ElMi09,YuLi06}.
\begin{theorem}\label{KLGL}
Consider \eqref{KLproblem} with $m = 1$ and $h$ defined as in \eqref{sec53h}. Suppose that Assumptions~\ref{assumptionxxx} and \ref{EllStrCon} hold and the origin is not feasible for \eqref{KLproblem}. Then $F$ is a KL function with exponent $\frac{1}{2}$.
\end{theorem}
\begin{proof}
As in the proof of Theorem~\ref{KLnorm}(iii), we again have $\Argmin F_{\bar \lambda}$ being compact and nonempty for any $\bar\lambda > 0$, where $F_{\bar{\lambda}}:= h + \bar{\lambda}\, q_1\circ A_1$. Fix any $\bar \lambda > 0$ and any $\bar{x} \in\Argmin F_{\bar{\lambda}}$, and note that $A_1x\equiv A_1\bar x$ over $\Argmin F_{\bar{\lambda}}$. Hence, $\bar\lambda A_1^T\nabla q_1(A_1x)\equiv \bar\lambda A_1^T\nabla q_1(A_1\bar x)=: \bar{\zeta}$ over $\Argmin F_{\bar{\lambda}}$. Moreover, we also have $-\bar{\zeta}\in\partial h(\bar{x})$ in view of \cite[Theorem~10.1]{RoWe98} and \cite[Exercise~8.8(c)]{RoWe98}. Then we define as before the sets
\[
\Gamma_f(y):= \left\{x\in\R^n:\; A_1x = y\right\} \text{  and  } \Gamma_{h}(\zeta):= \left\{x\in\R^n:\; -\zeta\in\partial h(x)\right\}.
\]
Since $\Gamma_{h}(\bar{\zeta}) = \prod_{J\in\mathcal{J}}\Gamma_{h_J}(\bar{\zeta}_J)$, and we have from Theorem~\ref{KLnorm}(i) that $\Gamma_{h_J}(\bar{\zeta}_J)$ is a polyhedron for every $J\in\mathcal{J}$, we deduce that $\Gamma_{h}(\bar{\zeta})$ is polyhedral. Thus, by \cite[Corollary~3]{BaBL99}, $\{\Gamma_f(A_1\bar{x}), \Gamma_{h}(\bar{\zeta})\}$ is boundedly linearly regular.

Next, since $-\bar{\zeta}\in\partial h(\bar{x})$, we have $-\bar{\zeta}_J\in \partial h_J(\bar{x}_J)$ for each $J\in {\cal J}$.
Now, Theorem~\ref{KLnorm}(ii) states that $\partial h_J$ is metrically subregular at $\bar{x}_J$ for $-\bar{\zeta}_J$, i.e., for each $J\in\mathcal{J}$, there exist constants $\eta_J, \rho_J > 0$ such that
\[
\d(x_J, (\partial h_J)^{-1}(-\bar{\zeta}_J))\leq \eta_J\d(-\bar{\zeta}_J, \partial h_J(x_J)) \text{ for all } x_J\in B(\bar{x}_J, \rho_J).
\]
Define $\rho = \min\limits_{J\in\mathcal{J}}\rho_J$. Then for all $x\in B(\bar{x}, \rho)$, we have from the above display that
\[
\begin{aligned}
\d^2\left(x, (\partial h)^{-1}(-\bar{\zeta})\right)&= \sum_{J\in\mathcal{J}} \d^2(x_J, \big(\partial h_J\big)^{-1}(-\bar{\zeta}_J))\\
&\leq \sum_{J\in\mathcal{J}} \eta_J^2\d^2(-\bar{\zeta}_J, \partial h_J(x_J))\leq \eta^2\d^2(-\bar{\zeta}, \partial h(x))),
\end{aligned}
\]
where $\eta = \max\limits_{J\in\mathcal{J}} \eta_J$. Thus, $\partial h$ is metrically subregular at $\bar{x}$ for $-\bar{\zeta}$.

Now, in view of \cite[Theorem~2]{ZhSo17}, \cite[Corollary~1]{ZhSo17} and the arbitrariness of $\bar{x}\in\Argmin F_{\bar{\lambda}}$, we conclude that $F_{\bar{\lambda}}$ satisfies the Luo-Tseng error bound. Then the desired conclusion follows from \cite[Theorem~4.1]{LiPong18} and Corollary~\ref{KL}. This completes the proof.
\end{proof}

\section{Numerical experiments}\label{sec6}

In this section, we consider two specific applications from signal processing that can be formulated as \eqref{P0} and \eqref{P1}, respectively. We then illustrate our algorithms on solving random instances of these problems.

\subsection{Group sparse signals with Gaussian noise}
In signal processing, when the signal is contaminated with Gaussian noise, the least squares loss function can be used \cite{CaRT06,ChDS01}. If it is also known that the original signal belongs to a union of subspaces, the group LASSO penalty \cite{YuLi06} can be used for inducing such structure in the recovered solution; see \cite{ElMi09}. Here, motivated by the used of the $\ell_{1-2}$ regularizer in compressed sensing in \cite{YiLH15}, we consider a natural extension that subtracts a positive multiple of the norm function from the group LASSO penalty function:
\begin{equation}\label{E1}
\begin{array}{rl}
\min\limits_{x\in\R^n} & \sum\limits_{J\in\mathcal{J}}\|x_J\| - \mu\|x\| \\
{\rm s.t.} & \|Ax - b\|\leq \sigma,
  \end{array}
\end{equation}
where $\mu \in (0,1)$, $A\in\R^{p\times n}$ has full row rank, $b\in\R^p$, $\sigma \in (0, \|b\|)$, $\mathcal{J}$ is a partition of $\{1,2,\cdots, n\}$, $x_J\in\R^{|J|}$ is the subvector of $x$ indexed by $J\in\mathcal{J}$. 

Since the feasible set of \eqref{E1} is unbounded, Algorithm~\ref{SQP_retract} cannot be directly applied to solving \eqref{E1}. Fortunately, one can show that \eqref{E1} is equivalent to the following model for some $M > 0$:
\begin{equation}\label{E3}
  \begin{array}{rl}
\min\limits_{x\in\R^n} & \sum\limits_{J\in\mathcal{J}}\|x_J\| - \mu\|x\| \\
{\rm s.t.} & \|Ax - b\|^2\leq \sigma^2,\\
           & \max\limits_{J\in \mathcal{J}}\|x_J\| \leq M.
  \end{array}
\end{equation}
Indeed, as $A$ has full row rank, we know that the feasible set of \eqref{E1} is nonempty because it contains $\tx := A^\dagger b$. Since the objective of \eqref{E1} is also level-bounded, the solution set of \eqref{E1}, denoted by $S$, is nonempty. Hence,
\[
S\subseteq \left\{x\in\R^n:\; \sum\limits_{J\in\mathcal{J}}\|x_J\| - \mu\|x\| \le \sum\limits_{J\in\mathcal{J}}\|\tx_J\| - \mu\|\tx\|\right\}.
\]
Thus, $x\in S$ implies $(1 - \mu)\max\limits_{J\in \mathcal{J}}\|x_J\| \le (1 - \mu)\|x\|\le \sum\limits_{J\in\mathcal{J}}\|x_J\| - \mu\|x\| \le \sum\limits_{J\in\mathcal{J}}\|\tx_J\| - \mu\|\tx\|$. This shows that \eqref{E3} has the same set of optimal solutions as \eqref{E1} if $M \ge (1 - \mu)^{-1}[\sum\limits_{J\in\mathcal{J}}\|\tx_J\| - \mu\|\tx\|]$.

With these choices of $M > 0$, problem \eqref{E3} is a special case of \eqref{P0} with $m = 1$, $P_1(x)= \sum\limits_{J\in\mathcal{J}}\|x_J\|$, $P_2(x) = \mu\|x\|$, $g_1(x) = \|Ax - b\|^2 - \sigma^2$ and $C = \{x\in\R^n:\; \max\limits_{J\in \mathcal{J}}\|x_J\| \leq M\}$. Moreover, Assumption~\ref{assumption1} holds for \eqref{E3} with $\xfeas = \tx$. Therefore, we can apply Algorithm~\ref{SQP_retract} to \eqref{E3}. In our numerical experiments below, we will be solving random instances of \eqref{E3} with $$M = (1 - \mu)^{-1}[\sum\limits_{J\in\mathcal{J}}\|\tx_J\| - \mu\|\tx\|].$$

\subsection{Compressed sensing with Cauchy noise for complex signals}\label{sec6.2}
{\color{black} In compressed sensing for real signals, when the measurement (say, of dimension $2p$) is contaminated with Cauchy noise, the Lorentzian norm $\|y\|_{LL_2,\gamma}:= \sum\limits_{j=1}^{2p}\log(1 + \gamma^{-2}y_j^2)$ can be employed as the loss function \cite{CaBA10,CaRA16}, where $\gamma > 0$ and $y\in \R^{2p}$. Here we explore the scenario when the signal and the measurement matrix have \emph{complex} entries with the transmission contaminated by Cauchy noise. Specifically, in this setting, an unknown sparse vector ${\bf z}_0\in {\mathbb{C}}^n$ is compressed via multiplication with a given matrix ${\bf A}\in \mathbb{C}^{p\times n}$ and the resulting ${\bf A} {\bf z}_0$ is transmitted. The target is to recover approximately the vector ${\bf z}_0$ from the received signal ${\bf b} := {\bf A}{\bf z}_0 + \varepsilon$, where $\varepsilon \in {\mathbb{C}}^p$ is the noise vector, with the real and imaginary parts of each entry following an i.i.d. Cauchy distribution with mean $0$. In view of the structure of $\varepsilon$, we will view it as a random vector in $\R^{2p}$ with i.i.d. Cauchy entries and consider the Lorentzian norm for this real vector in our compressed sensing model below.

To derive the model, we start by writing
\[
{\bf A} := {\bf A}_{\rm re} + i {\bf A}_{\rm im}, \ \ {\bf b} := {\bf b}_{\rm re} + i {\bf b}_{\rm im},
\]
where ${\bf A}_{\rm re}$, ${\bf A}_{\rm im}\in \R^{p\times n}$, and ${\bf b}_{\rm re}$, ${\bf b}_{\rm im}\in \R^{p}$, and define
\begin{equation}\label{defineAb}
A := \begin{bmatrix}
{\bf A}_{\rm re} & -{\bf A}_{\rm im}\\
{\bf A}_{\rm im} & {\bf A}_{\rm re}
\end{bmatrix}\in \R^{2p\times 2n},\ \ \ b:= \begin{bmatrix}
  {\bf b}_{\rm re}\\{\bf b}_{\rm im}
\end{bmatrix}\in \R^{2p}.
\end{equation}
Then, notice that for any complex vector ${\bf z} = {\bf x} + i {\bf y}\in {\mathbb{C}}^n$, where ${\bf x}$ and ${\bf y}$ are respectively the real part and imaginary part of ${\bf z}$, it holds that
\[
{\bf A}{\bf z} - {\bf b} = {\bf A}_{\rm re}{\bf x} - {\bf A}_{\rm im}{\bf y} - {\bf b}_{\rm re} + i [{\bf A}_{\rm im}{\bf x} + {\bf A}_{\rm re}{\bf y} - {\bf b}_{\rm im}].
\]
Thus, we can consider the following loss function for the Cauchy measurement noise:
\[
\left\|
\begin{bmatrix}
  {\bf A}_{\rm re}{\bf x} - {\bf A}_{\rm im}{\bf y} - {\bf b}_{\rm re}\\
  {\bf A}_{\rm im}{\bf x} + {\bf A}_{\rm re}{\bf y} - {\bf b}_{\rm im}
\end{bmatrix} \right\|_{LL_2,\gamma} \!\!\!\!\! = \|Ax - b\|_{LL_2,\gamma},\ \ {\rm where}\ \ x := \begin{bmatrix}
  {\bf x}\\{\bf y}
\end{bmatrix},
\]
and $\gamma > 0$ is an estimate on the standard deviation of the Cauchy distribution.
Based on the above observation and motivated by the use of $\ell_{1-2}$-type regularizer in \cite{YiLH15} for inducing sparsity, we consider the following model for recovering sparse signal with entries in complex numbers from measurements contaminated with Cauchy measurement noise:
\begin{equation}\label{E2}
  \begin{array}{rl}
\min\limits_{x\in\R^{2n}} & \sum\limits_{J\in\mathcal{J}}\|x_J\| - \mu\|x\| \\
{\rm s.t.} & \|Ax - b\|_{LL_2,\gamma}\leq \overline{\sigma},
  \end{array}
\end{equation}
where $A\in\R^{2p\times 2n}$ and $b\in \R^{2p}$ are given in \eqref{defineAb}, $\mu \in (0,1)$ and $\overline{\sigma} \in (0,\|b\|_{LL_2,\gamma})$, $\mathcal{J} := \bigcup_{i=1}^n\mathcal{J}_i$ with $\mathcal{J}_i := \{i,i+n\}$ for each $i$, and $x_J\in\R^2$ is the subvector of $x$ indexed by $J\in\mathcal{J}$. Moreover, we assume without loss of generality that $A$ has full row rank: this can be inferred from the fact that the measurement matrix ${\bf A}$ typically has full row rank. 

Although the feasible region of \eqref{E2} is unbounded and is not a special case of \eqref{P1}, one can argue as in the discussion following \eqref{E3} that \eqref{E2} is equivalent to the following model:
\begin{equation}\label{E4}
  \begin{array}{rl}
\min\limits_{x\in\R^{2n}} & \sum\limits_{J\in\mathcal{J}}\|x_J\| - \mu\|x\| \\
{\rm s.t.} & \|Ax - b\|_{LL_2,\gamma}\leq \overline{\sigma},\\
           & \max\limits_{J\in \mathcal{J}}\|x_J\| \leq M_1
  \end{array}
\end{equation}
as long as $M_1 \ge (1 - \mu)^{-1}[\sum\limits_{J\in\mathcal{J}}\|(A^\dagger b)_J\| - \mu\|A^\dagger b\|]$.
Notice that \eqref{E4} is a special case of \eqref{P1} with $P_1(x) = \sum\limits_{J\in\mathcal{J}}\|x_J\|$, $P_2(x) = \mu\|x\|$,  $\sigma=\overline{\sigma}$, $\ell(y) = \|y\|_{LL_2,\gamma}$, and $C = \{x\in\R^{2n}:\; \max\limits_{J\in \mathcal{J}}\|x_J\| \leq M_1\}$. One can see  from Example \ref{Loren} that Assumption~\ref{gassum} is satisfied for \eqref{E4}, and Assumption~\ref{assumption4} is satisfied for \eqref{E4} with $\xfeasss = A^\dagger b$. Hence, we can apply Algorithm \ref{SQP_retract2} to \eqref{E4}. In our numerical experiments below, we will be solving random instances of \eqref{E2} with $$M_1 = (1 - \mu)^{-1}[\sum\limits_{J\in\mathcal{J}}\|(A^\dagger b)_J\| - \mu\|A^\dagger b\|].$$ }

\subsection{Numerical results}
In this subsection, we perform numerical experiments to study the performance of Algorithm~\ref{SQP_retract} on solving \eqref{E3} and that of Algorithm~\ref{SQP_retract2} on solving \eqref{E4}.
We compare our approaches with Algorithm~\ref{ESQM} below, which is an adaptation of the ESQM with line search (ESQM$_{\rm ls}$) from \cite{Au13}. The ESQM$_{\rm ls}$ in \cite{Au13} was originally designed for \eqref{P0} with $P$ and all the $g_i$'s being twice continuously differentiable on $C$, and subsequential convergence was established under suitable assumptions. Here, we adapt the ESQM$_{\rm ls}$ as Algorithm~\ref{ESQM} for our problems. Note that the convergence analysis in \cite{Au13} can be readily adapted to establish subsequential convergence of Algorithm~\ref{ESQM} for solving \eqref{E3}. Similar to our algorithms, Algorithm~\ref{ESQM} also involves affine approximations in each subproblem (see \eqref{subproblemhaha}). However, unlike our approaches, it does not force all iterates to be feasible.

\begin{algorithm}
\caption{(Adaptation of) ESQM$_{\rm ls}$ in \cite{Au13} for \eqref{P0}}\label{ESQM}
\begin{algorithmic}
\STATE
\begin{description}
  \item[\bf Step 0.] Choose $x^0\in C$, $c \in (0, 1)$, $\beta_0 > 0$, $t\in(0, 1)$ and $\delta > 0$. Set $k = 0$.
  \item[\bf Step 1.] Pick any $\xi^k\in\partial P_2(x^k)$ and compute
  \begin{equation}\label{subproblemhaha}
  \begin{array}{rl}
  (\widetilde u,\widetilde s) = \argmin\limits_{(x,s)\in\R^{n+1}} & P_1(x) -\langle \xi^k, x - x^k\rangle + \frac{1}{\beta_k} s + \frac{1}{2\beta_k}\|x - x^k\|^2\\ [2pt]
      {\rm s.t.}& g_i(x^k) + \langle \nabla g_i(x^k), x - x^k\rangle \leq s,\ i = 1, \ldots, m,\\ [2pt]
      & s \geq 0,\ x\in C.
  \end{array}
  \end{equation}

  \item[\bf Step 2.] Let $t_k$ be the largest element in $\{1, t, t^2, \ldots\}$ so that
  \begin{equation*}
    F_k(x^k + t_k (\widetilde{u} - x^k)) \le F_k(x^k) - c t_k \beta_k^{-1} \|\widetilde{u} - x^k\|^2
  \end{equation*}
  where $F_k(x):= P(x) + \beta^{-1}_k \max\{\widehat g(x),0\}$ and $\widehat g(x) := \max\limits_{1\le i\le m}\{g_i(x)\}$.

  \item[\bf Step 3.] If $g_i(x^k) + \langle \nabla g_i(x^k), \widetilde{u} - x^k\rangle \leq 0$ for all $i$, then $\beta_{k+1} = \beta_k$, otherwise $\beta_{k+1}^{-1} = \beta_k^{-1} + \delta$. Set $u^k = \widetilde u$ and $x^{k+1} = x^k + t_k (\widetilde{u} - x^k)$. Update $k \leftarrow k+1$ and go to \textbf{Step 1}.
\end{description}
\end{algorithmic}
\end{algorithm}

We next present the parameters, initial points and termination criteria used in our numerical experiments. Solution approaches for the subproblems involved are outlined in the appendix. All numerical experiments are performed in {\color{black} MATLAB R2016a on a 64-bit PC with an Intel(R) Core(TM) i7-10710U CPU (@1.10GHz, 1.61GHz) and 16GB of RAM.}

\textbf{Parameter settings:} In Algorithms~\ref{SQP_retract} and \ref{SQP_retract2}, we let $c = 10^{-4}$, $\eta = \frac{1}{2}$, $\underline{\beta} = 10^{-8}$, $\overline{\beta} = 10^8$. We set $\beta^0_0=1$ and, for $k \geq 1$, we choose
\[
\beta_k^0 = \begin{cases}
  \min\{\max\{\underline\beta,2\beta_{k-1}^0\},\overline\beta\} & {\rm if}\ \beta_{k-1} = \beta_{k-1}^0,\\
  \min\{\max\{\underline\beta,\beta_{k-1}\},\overline\beta\} & {\rm otherwise.}
\end{cases}
\]

In Algorithm~\ref{ESQM}, we let $c = 10^{-4}$, $\beta_0 = 1$, $t = \frac{1}{2}$. It appears that the performance of ESQM$_{\rm ls}$ can be sensitive to the choice of $\delta$. In our experiments below, we consider 3 choices of $\delta$: $0.5$, $0.1$ or $0.02$.

\textbf{Slater point:} We take the Slater points as $\xfeas = \xfeasss = A^\dagger b$, and compute $A^\dagger b$ via the MATLAB commands \verb+[Q,R]=qr(A',0); x=Q*(R'\b)+.

\textbf{Initialization:} For \eqref{E3}, we first solve \eqref{E1} with $\mu = 0$ approximately using SPGL1 \cite{BeFr09} (version 2.1), under default settings, to obtain an approximate solution $x_{\rm spgl1}$. We then project $x_{\rm spgl1}$ onto the box $\{x:\; \max\limits_{J\in{\cal J}}\|x_J\|\le M\}$ to obtain $\widehat x_{\rm spgl1}$. The initial point $x^0$ is then generated as follows:
\begin{equation}\label{initial1}
x^0 = \begin{cases}
  \widehat x_{\rm spgl1} + \tau(A^\dagger b - \widehat x_{\rm spgl1}) & {\rm if}\ \|A\widehat x_{\rm spgl1} - b\| > \sigma,\\
  \widehat x_{\rm spgl1} & {\rm otherwise},
\end{cases}
\end{equation}
where $\tau\in(0, 1)$ is a positive root of the quadratic equation $\|A(\widehat x_{\rm spgl1} + \tau(A^\dagger b - \widehat x_{\rm spgl1})) - b\|^2=\sigma^2$.

For \eqref{E4}, we first find a positive constant $\tau_0\in(0, 1)$ so that $\|b - A(\tau_0 A^\dagger b)\|_{LL_2,\gamma} = \overline\sigma$.\footnote{Note that $b - A(\tau_0 A^\dagger b) = (1 - \tau_0)b$, and hence such a $\tau_0$ can be found efficiently by first applying the Newton's method to obtain a root $s_*$ of the function $s\mapsto \|\sqrt{s}b\|_{LL_2,\gamma} - \overline\sigma$ (initialized at $s = 0$) and then setting $\tau_0 = 1 - \sqrt{s_*}$.}
Then we apply the SPGL1 \cite{BeFr09} (version 2.1) with default settings to approximately solve the following convex problem to obtain an approximate solution $x_{\rm spgl1}$:
\begin{equation}\label{SPGL1nonconvexL1}
\begin{array}{rl}
  \min\limits_{x\in \R^n} & \sum\limits_{J\in\mathcal{J}}\|x_J\| \\
  {\rm s.t.} & \ell^{\tau_0A^\dagger b}(Ax - b)\le \widetilde\sigma^{\tau_0A^\dagger b},
\end{array}
\end{equation}
where $\ell^{\tau_0A^\dagger b}$ and $\widetilde\sigma^{\tau_0A^\dagger b}$ are defined as in \eqref{defell} and \eqref{defsigma} respectively, with $\tau_0A^\dagger b$ in place of $y$ there. We then project $x_{\rm spgl1}$ onto the box $\{x:\; \|x\|_\infty\le M_1\}$ to obtain $\widetilde x_{\rm spgl1}$. The initial point $x^0$ is then generated as follows:
\begin{equation}\label{initial2}
x^0 = \begin{cases}
  \widetilde x_{\rm spgl1} + \tau_1(A^\dagger b - \widetilde x_{\rm spgl1}) & {\rm if}\ \|A\widetilde x_{\rm spgl1} - b\|_{LL_2,\gamma} > \overline\sigma,\\
  \widetilde x_{\rm spgl1} & {\rm otherwise},
\end{cases}
\end{equation}
where $\tau_1\in (0, 1)$ is chosen so that $\left\|b - A\left[\widetilde x_{\rm spgl1} + \tau_1(A^\dagger b - \widetilde x_{\rm spgl1})\right]\right\|_{LL_2,\gamma} = \overline\sigma$.\footnote{Note that $b - A[\widetilde x_{\rm spgl1} + \tau_1(A^\dagger b - \widetilde x_{\rm spgl1})] = (1 - \tau_1)(b-A \widetilde x_{\rm spgl1})$} One can find such a constant $\tau_1$ efficiently by first applying Newton's method to obtain a root $s_*$ of the function $s\mapsto \|\sqrt{s}(b - A\widetilde x_{\rm spgl1})\|_{LL_2,\gamma} - \overline\sigma$ (initialized at $s = 0$) and then setting $\tau_1 = 1 - \sqrt{s_*}$.

\textbf{Termination criteria.} The termination criteria are based on verifying approximately conditions (i), (ii) and (iii) in Definition~\ref{Stationary} by $x = u^k$ and $\lambda = \lambda_k$, where $u^k$ comes from the solution of the subproblems \eqref{subproblem2}, \eqref{subproblem4} or \eqref{subproblemhaha}, and $\lambda_k$ is a Lagrange multiplier corresponding to the affine inequality constraint induced from $g_1$.

Specifically, the algorithms are terminated when
\begin{equation}\label{termination1}
\max\{\|\xi_u^k - \xi^k\| + {\frak L}_k\|u^k - x^k\|, 10^2\cdot\max\{|\lambda_k g_1(u^k)|, g_1(u^k)\}\} \leq 10^{-4}\cdot\max\{\|u^k\|, 1\},
\end{equation}
where $\xi^k_u = \argmin\limits_{\xi\in \mu\partial \|u^k\|}\{\|\xi\|\}$ and
\[
{\frak L}_k =
\begin{cases}
  2\lambda_k \|A\|^2 + \beta_k^{-1} & \mbox{for \eqref{E3}},\\
  2\lambda_k \|A\|^2\gamma^{-2} + \beta_k^{-1} & \mbox{for \eqref{E4}}.
\end{cases}
\]
Moreover, as a safeguard, we also terminate when stepsizes are \emph{small}. Specifically, we terminate Algorithms~\ref{SQP_retract} and \ref{SQP_retract2} when $\beta_k \le 10^{-10}$, and terminate Algorithm~\ref{ESQM} when $t_k \le 10^{-10}$. Note that when \eqref{termination1} is satisfied, we have $g_1(u^k)\le 10^{-6}\cdot\max\{\|u^k\|, 1\}$ and $|\lambda_k g_1(u^k)|\le 10^{-6}\cdot\max\{\|u^k\|, 1\}$, and we can deduce from the KKT conditions of subproblems~\eqref{subproblem2}, \eqref{subproblem4} and \eqref{subproblemhaha} that $\d(0, \partial F_k(u^k)) \leq 10^{-4}\cdot\max\{\|u^k\|, 1\}$ when $u^k\neq 0$, where $F_k(x) = P_1(x) - P_2(x) + \lambda_kg_1(x) + \delta_C(x)$.

\subsubsection{ Group LASSO with Gaussian noise. }
We compare Algorithm~\ref{SQP_retract} and Algorithm~\ref{ESQM} on solving random instances of \eqref{E3}. Specifically, we generate an $A\in \R^{p\times n}$ with i.i.d. standard Gaussian entries and then normalize each column of $A$. Next, we use the following MATLAB command to generate the original block-sparse signal $\xorig\in\R^n$:
\begin{verbatim}
         I = randperm(n/j); I = I(k+1:end); x0 = randn(j,n/j);
         x0(:,I) = 0;  x_orig = reshape(x0,n,1),
\end{verbatim}
where $j$ is the size of each block and $k$ is the number of nonzero blocks. We let $b = A\xorig + 0.005\bar{\epsilon}$, where $\bar{\epsilon}$ is a random vector which has i.i.d standard Gaussian entries. Finally, we set $\sigma = 1.2\|0.005\bar{\epsilon}\|$.


In our numerical tests, we let $\mu =0.95$, {\color{black}$j = 2$, $(p,n,k) = (720i,2560i,120i)$} with $i\in \{2, 4, 6, 8, 10\}$. For each $i$, we generate {\color{black}20 random data} as described above. We present the computational results in Table~\ref{table1}, averaged over the {\color{black}$20$ random instances}. Here, Algorithm~\ref{SQP_retract} is denoted by FPA$_{\rm retract}$, Algorithm~\ref{ESQM} with different $\delta$ is denoted by ESQM$_{\rm ls}$ with the $\delta$ specified, and recall that SPGL1 is the algorithm used for obtaining initial points for FPA$_{\rm retract}$ and ESQM$_{\rm ls}$. We present the time for computing the QR decomposition of $A^T$ (denoted by QR), the time for computing $A^\dagger b$ given the QR factorization of $A^T$ (denoted by Slater), the CPU times of the algorithms (CPU time),\footnote{The reported CPU times of FPA$_{\rm retract}$ and ESQM$_{\rm ls}$ do not include the time for generating $\xfeas$ and $x^0$.} the number of iterations (denoted by Iter), the recovery errors $\text{RecErr} := \frac{\|x^* - \xorig\|}{\max\{1, \|\xorig\|\}}$ and the residuals ${\rm Residual} := \frac{\|Ax^* - b\| - \sigma }{\sigma}$, where $x^*$ is the approximate solution returned by the respective algorithm.

From Table ~\ref{table1}, one can see that the recovery errors of FPA$_{\rm retract}$ and ESQM$_{\rm ls}$ are comparable, and FPA$_{\rm retract}$ is usually faster. Moreover, the performance of ESQM$_{\rm ls}$ appears to be quite sensitive to the choice of $\delta$. {\color{black}On passing, we also point out that the solutions returned by FPA$_{\rm retract}$ and ESQM$_{\rm ls}$ tend to have smaller residuals (in absolute value) than those obtained by SPGL1.}


\begin{table}[h]
{\color{black}
\caption{Computational results for problem \eqref{E3}}\label{table1}
\begin{center}
{\footnotesize
\begin{tabular}{|c|c|c|c|c|c|c|}\hline
\phantom{\diagbox{Date}{$i$}} & \multicolumn{1}{|c|}{Method} & \multicolumn{1}{|c|}{$i = 2$} & \multicolumn{1}{c|}{ $i = 4$ }
& \multicolumn{1}{c|}{ $i = 6$ } & \multicolumn{1}{|c|}{ $i = 8$ } & \multicolumn{1}{c|}{ $i = 10$ }\\\cline{1-7}
\multirow{7}*{CPU time} & \multirow{1}*{QR}
&   0.55 &   2.64 &   8.27 &  21.07 &  40.90\\\cline{2-2} \multirow{1}*{} & \multirow{1}*{Slater}
&   0.01 &   0.02 &   0.05 &   0.09 &   0.14\\\cline{2-2} \multirow{1}*{} & \multirow{1}*{SPGL1}
&   1.37 &   6.52 &  16.54 &  29.97 &  51.87\\\cline{2-7} \multirow{1}*{} & \multirow{1}*{FPA$_{\rm retract}$}
&   2.62 &  11.00 &  31.82 &  57.61 &  82.91\\\cline{2-2} \multirow{1}*{}  & \multirow{1}*{ESQM$_{{\rm ls},\ \delta = 0.5}$}
&   4.36 &  21.02 &  57.37 & 102.78 & 162.15\\\cline{2-2} \multirow{1}*{}  & \multirow{1}*{ESQM$_{{\rm ls},\ \delta = 0.1}$}
&   3.99 &  18.85 &  49.01 &  87.80 & 139.50\\\cline{2-2} \multirow{1}*{}  & \multirow{1}*{ESQM$_{{\rm ls},\ \delta = 0.02}$}
&  11.87 &  54.43 & 132.39 & 235.45 & 378.77\\\cline{1-7} \multirow{4}*{Iter} & \multirow{1}*{FPA$_{\rm retract}$}
&    342 &    349 &    449 &    460 &    421\\\cline{2-2} \multirow{1}*{}     & \multirow{1}*{ESQM$_{{\rm ls},\ \delta = 0.5}$}
&    756 &    899 &   1094 &   1111 &   1120\\\cline{2-2} \multirow{1}*{}     & \multirow{1}*{ESQM$_{{\rm ls},\ \delta = 0.1}$}
&    774 &    853 &    964 &    972 &    976\\\cline{2-2} \multirow{1}*{}     & \multirow{1}*{ESQM$_{{\rm ls},\ \delta = 0.02}$}
&   2447 &   2536 &   2650 &   2648 &   2665\\\cline{1-7} \multirow{5}*{RecErr} & \multirow{1}*{SPGL1}
&  0.054 &  0.044 &  0.053 &  0.054 &  0.045\\\cline{2-7} \multirow{1}*{} & \multirow{1}*{FPA$_{\rm retract}$}
&  0.030 &  0.032 &  0.035 &  0.035 &  0.035\\\cline{2-2} \multirow{1}*{} & \multirow{1}*{ESQM$_{{\rm ls},\ \delta = 0.5}$}
&  0.030 &  0.032 &  0.035 &  0.035 &  0.035\\\cline{2-2} \multirow{1}*{} & \multirow{1}*{ESQM$_{{\rm ls},\ \delta = 0.1}$}
&  0.030 &  0.032 &  0.035 &  0.035 &  0.035\\\cline{2-2} \multirow{1}*{} & \multirow{1}*{ESQM$_{{\rm ls},\ \delta = 0.02}$}
&  0.030 &  0.032 &  0.035 &  0.035 &  0.035\\\cline{1-7} \multirow{5}*{Residual} & \multirow{1}*{SPGL1}
& -2.17e-04 & -1.90e-04 & -1.68e-04 & -1.47e-04 & -1.27e-04\\\cline{2-7} \multirow{1}*{} & \multirow{1}*{FPA$_{\rm retract}$}
& -1.48e-15 & 6.36e-16 & -1.46e-16 & 2.41e-15 & -1.04e-15\\\cline{2-2} \multirow{1}*{} & \multirow{1}*{ESQM$_{{\rm ls},\ \delta = 0.5}$}
& 1.13e-10 & 1.24e-10 & 1.10e-10 & 1.13e-10 & 1.13e-10\\\cline{2-2} \multirow{1}*{}      & \multirow{1}*{ESQM$_{{\rm ls},\ \delta = 0.1}$}
& 9.71e-11 & 1.01e-10 & 9.62e-11 & 1.01e-10 & 1.02e-10\\\cline{2-2} \multirow{1}*{}      & \multirow{1}*{ESQM$_{{\rm ls},\ \delta = 0.02}$}
& 1.06e-10 & 9.77e-11 & 9.87e-11 & 9.51e-11 & 9.51e-11\\\cline{1-7}
\end{tabular}
}
\end{center}
}
\end{table}

\subsubsection{ Compressed sensing with Cauchy noise}

{\color{black}We compare Algorithm~\ref{SQP_retract2} and Algorithm~\ref{ESQM} on solving random instances of \eqref{E4}. Specifically, we generate two matrices ${\bf A}_{\rm re}\in \R^{p\times n}$ and ${\bf A}_{\rm im}\in \R^{p\times n}$ with i.i.d standard Gaussian entries, let
\begin{equation*}
A := \left[ \begin{matrix}
{\bf A}_{\rm re} & -{\bf A}_{\rm im}\\
{\bf A}_{\rm im} & {\bf A}_{\rm re}
\end{matrix} \right]
\end{equation*}
and then normalize each column of $A$. Next, we generate two vectors ${\bf u}\in\R^k$ and ${\bf v}\in\R^k$ with i.i.d standard Gaussian entries. We choose an index set $I$ of size $k$ at random and define the original sparse signal ${\bf z}_{\rm orig}\in {\mathbb{C}}^{n}$ by setting ${\bf z}_j := {\bf u}_j + i {\bf v}_j$ for all $j \in I$ and ${\bf z}_j = 0$ otherwise. We then let
\[
\xorig = \begin{bmatrix}
  {\bf z}_{\rm re}\\
  {\bf z}_{\rm im}
\end{bmatrix}\ \ {\rm and}\ \ b = A\xorig + 0.005\hat{\epsilon},
\]
where ${\bf z}_{\rm re}$ and ${\bf z}_{\rm im}$ are respectively the real and imaginary parts of ${\bf z}_{\rm orig}$, and $\hat{\epsilon}$ is a random vector with each entry independently following the Cauchy$(0, 1)$ distribution, i.e., $\hat{\epsilon}_{j} = \tan(\pi(\widetilde{\epsilon}_j - \frac{1}{2}))$ with each $\widetilde{\epsilon}_j\in\R^{2p}$ being uniformly chosen in $[0, 1]$. Finally, we set $\overline{\sigma} = 1.2\|0.005\hat{\epsilon}\|_{LL_2,\gamma}$ with $\gamma = 0.05$.}


In our numerical experiments, we let $\mu =0.95$, $j = 2$, $\gamma = 0.05$, {\color{black} $(p,n,k) = (360i,1280i,60i)$ }with $i\in \{2, 4, 6, 8, 10\}$. For each $i$, we generate {\color{black}20 random instances} as described above. The computational results averaged over the {\color{black} $20$ random instances} are shown in Table~\ref{table2}. In the table, Algorithm~\ref{SQP_retract2} is denoted by FPA$_{\rm retract}$, Algorithm~\ref{ESQM} with different $\delta$ is denoted by ESQM$_{\rm ls}$ with the $\delta$ specified, and recall that SPGL1 is the algorithm used for obtaining initial points for FPA$_{\rm retract}$ and ESQM$_{\rm ls}$ by solving the auxiliary problem~\eqref{SPGL1nonconvexL1}. As before, we present the time for performing QR decomposition on $A^T$ (denoted by QR), the time for computing $A^\dagger b$ given the QR factorization of $A^T$ (denoted by Slater), the CPU time of algorithms (CPU time),\footnote{The reported CPU times of FPA$_{\rm retract}$ and ESQM$_{\rm ls}$ do not include the time for generating $\xfeasss$ and $x^0$.} the number of iterations (denoted by Iter), the recovery errors $\text{RecErr} := \frac{\|x^* - \xorig\|}{\max\{1, \|\xorig\|\}}$ and the residuals ${\rm Residual} := \frac{\|Ax^* - b\|_{LL_2,\gamma} - \overline{\sigma} }{\overline{\sigma}}$, where $x^*$ is the approximate solution returned by the respective algorithm.

From Table~\ref{table2}, one can see that the recovery errors of FPA$_{\rm retract}$ and ESQM$_{\rm ls}$ are comparable, and FPA$_{\rm retract}$ is usually faster, except for the case when {\color{black}$i = 6$}.
Indeed, we observe that the performance of FPA$_{\rm retract}$ can be quite sensitive to the ``quality" of the initial point (reflected partly by the RecErr of the approximate solution returned by SPGL1).


\begin{table}[h]
{\color{black}
\caption{Computational results for problem \eqref{E4}}\label{table2}
\begin{center}
{\footnotesize
\begin{tabular}{|c|c|c|c|c|c|c|}\hline
\phantom{\diagbox{Date}{$i$}} & \multicolumn{1}{|c|}{Method} & \multicolumn{1}{|c|}{$i = 2$} & \multicolumn{1}{c|}{ $i = 4$ } & \multicolumn{1}{c|}{ $i = 6$ } & \multicolumn{1}{|c|}{ $i = 8$ } & \multicolumn{1}{c|}{ $i = 10$ }\\\cline{1-7}
\multirow{7}*{CPU time} & \multirow{1}*{QR}
&   0.61 &   2.82 &   8.42 &  21.15 &  40.40 \\\cline{2-2} \multirow{1}*{} & \multirow{1}*{Slater}
&   0.01 &   0.02 &   0.05 &   0.10 &   0.15 \\\cline{2-2} \multirow{1}*{} & \multirow{1}*{SPGL1}
&   0.97 &  12.69 &  68.82 &  22.22 &  98.06 \\\cline{2-7} \multirow{1}*{} & \multirow{1}*{FPA$_{\rm retract}$}
&   2.58 &  13.57 & 424.98 &  35.03 &  73.58 \\\cline{2-2} \multirow{1}*{} & \multirow{1}*{ESQM$_{{\rm ls},\ \delta = 0.5}$}
&  13.16 &  38.65 &  69.74 & 102.27 & 132.12 \\\cline{2-2} \multirow{1}*{} & \multirow{1}*{ESQM$_{{\rm ls},\ \delta = 0.1}$}
&  14.25 &  46.93 & 104.76 & 177.87 & 189.31 \\\cline{2-2} \multirow{1}*{} & \multirow{1}*{ESQM$_{{\rm ls},\ \delta = 0.02}$}
&  13.83 &  47.29 & 102.40 & 169.08 & 266.92 \\\cline{1-7} \multirow{4}*{Iter} & \multirow{1}*{FPA$_{\rm retract}$}
&    221 &    368 &   5611 &    272 &    371 \\\cline{2-2} \multirow{1}*{}     & \multirow{1}*{ESQM$_{{\rm ls},\ \delta = 0.5}$}
&   1745 &   1480 &   1229 &   1072 &    896 \\\cline{2-2} \multirow{1}*{}     & \multirow{1}*{ESQM$_{{\rm ls},\ \delta = 0.1}$}
&   1888 &   1796 &   1862 &   1877 &   1288 \\\cline{2-2} \multirow{1}*{}     & \multirow{1}*{ESQM$_{{\rm ls},\ \delta = 0.02}$}
&   1828 &   1814 &   1824 &   1787 &   1823 \\\cline{1-7} \multirow{5}*{RecErr} & \multirow{1}*{SPGL1}
&  0.869 &  1.494 &  5.501 &  0.964 &  1.443 \\\cline{2-7} \multirow{1}*{}& \multirow{1}*{FPA$_{\rm retract}$}
&  0.048 &  0.051 &  0.051 &  0.051 &  0.052 \\\cline{2-2} \multirow{1}*{}& \multirow{1}*{ESQM$_{{\rm ls},\ \delta = 0.5}$}
&  0.048 &  0.051 &  0.051 &  0.051 &  0.052 \\\cline{2-2} \multirow{1}*{}& \multirow{1}*{ESQM$_{{\rm ls},\ \delta = 0.1}$}
&  0.048 &  0.051 &  0.051 &  0.051 &  0.052 \\\cline{2-2} \multirow{1}*{}& \multirow{1}*{ESQM$_{{\rm ls},\ \delta = 0.02}$}
&  0.048 &  0.051 &  0.051 &  0.051 &  0.052 \\\cline{1-7} \multirow{5}*{Residual} & \multirow{1}*{SPGL1}
& -5.23e-01 & -6.27e-01 & -7.35e-01 & -6.01e-01 & -6.86e-01 \\\cline{2-7} \multirow{1}*{} & \multirow{1}*{FPA$_{\rm retract}$}
& -2.86e-12 & -2.37e-12 & -2.19e-12 & -2.45e-12 & -2.32e-12 \\\cline{2-2} \multirow{1}*{} & \multirow{1}*{ESQM$_{{\rm ls},\ \delta = 0.5}$}
& 7.20e-11 & 1.05e-10 & 8.33e-11 & 2.08e-10 & 1.11e-10 \\\cline{2-2} \multirow{1}*{}      & \multirow{1}*{ESQM$_{{\rm ls},\ \delta = 0.1}$}
& 2.95e-11 & 3.25e-11 & 3.15e-11 & 3.67e-11 & 9.77e-11 \\\cline{2-2} \multirow{1}*{}      & \multirow{1}*{ESQM$_{{\rm ls},\ \delta = 0.02}$}
& 2.70e-11 & 2.53e-11 & 2.76e-11 & 2.58e-11 & 2.93e-11 \\\cline{1-7}
\end{tabular}
}
\end{center}
}
\end{table}

{\color{black} \section{Conclusions}
In this paper, we considered difference-of-convex (DC) programs
with smooth inequality and simple geometric constraints, and proposed first-order feasible methods for solving this class of nonconvex optimization problems. Our proposed methods were based on affine approximations for the constraints, and maintained the feasibility of the iterates by using a ``retractation step" inspired from the development of manifold optimization. We then established the global subsequential convergence of the sequence generated by the proposed algorithms under strict feasibility condition if the constraint functions are convex. We also showed that our methods can be extended to solve DC
programs with a single specially structured nonconvex constraint. Finally, we illustrated the promising numerical performance of the proposed methods via two concrete optimization models: group-structured compressed sensing problems with Gaussian measurement noise and compressed sensing problems with Cauchy measurement noise.

Our results in this paper point out several interesting research questions. For example, in the case where the constraint functions are convex, we established that the sequence generated by our method is bounded, and any cluster point of the generated sequence is a critical point of the underlying optimization problem. It would be interesting to see whether the convexity assumption of the constraint functions can be relaxed as some suitable generalized convexity conditions such as quasi-convexity. Moreover, for nonconvex constrained DC programs, we showed that our method can be extended to solve DC programs with a single specially structured nonconvex constraint. Although this nontrivial extension allows us to cover some important applications such as compressed sensing problems with Cauchy measurement noise, how to further extend this to general nonconvex problems with multiple constraints is also an interesting and challenging research question to pursue. Finally,  \cite{PRA17} discussed the so-called d-stationary points for DC programs which is a stronger notion than the critical points, and designed algorithms with convergence gurantee to d-stationary points. Another interesting research direction is to examine how our methods can be modified to give rise to convergence to d-stationary points. These are some potential future research topics which deserve further study.}

\appendix

\section{Subproblems in FPA$_{\rm retract}$ and ESQM$_{\rm ls}$ for \eqref{E3} and \eqref{E4}}

In this appendix, we outline how the subproblems \eqref{subproblem2}, \eqref{subproblem4} and \eqref{subproblemhaha} are solved in our numerical experiments, which solve random instances of \eqref{E3} and \eqref{E4}.

\subsection{Solving the subproblem in FPA$_{\rm retract}$ for \eqref{E3} and \eqref{E4} }\label{sec:append1}
The subproblems \eqref{subproblem2} (with $m = 1$) and \eqref{subproblem4} can be reformulated as follows:
\begin{equation}\label{solvsub}
  \begin{array}{rl}
\min\limits_{x\in\R^n} & P_1(x) + \frac{1}{2\beta}\|x - y\|^2 + \delta_{C}(x)\\
{\rm s.t.} & \langle a, x\rangle\leq r,
  \end{array}
\end{equation}
where $C$ is a compact convex set, $y$, $a\in\R^n$, $r\in \R$ and $\beta >0$. Moreover, the Slater condition holds thanks to Theorem~\ref{welldef1}(i) and Theorem~\ref{welldef2}(i).

By \cite[Corollary~28.2.1, Theorem~28.3]{Ro70}, $x^*$ is optimal for \eqref{solvsub} if and only if $x^*$ is feasible {\color{black} for \eqref{solvsub}} and there exists $\lambda_*\ge 0$ such that the following conditions hold:
\begin{enumerate}[{\rm (I)}]
  \item\label{subKKT}  $0\in\partial (P_1 + \delta_{C})(x^*) + \frac{1}{\beta}(x^* - y) + \lambda_*a $;
  \item\label{subKKT2}  $\lambda_*(\langle a, x^*\rangle - r)=0$.
\end{enumerate}
Note that \eqref{subKKT} is equivalent to
\begin{equation}\label{solution}
x^* = \argmin_{x\in C}\left\{\beta P_1(x) + \frac12\|y - \lambda_*\beta a - x\|^2\right\}.
\end{equation}
For \eqref{E3}, since $P_1(x) = \sum\limits_{J\in\mathcal{J}}\|x_J\|$ and $C = \{x:\; \max\limits_{J\in {\cal J}}\|x_J\|\leq M\}$, through an argument similar to the one outlined in \cite[Appendix]{LiuPo17}, the relation \eqref{solution} is equivalent to
\begin{equation}\label{solution1}
x^*_J = \min\left\{\max\left\{1 - \frac{\beta}{\|y_J - \lambda_*\beta a_J\|}, 0\right\},\frac{M}{\|y_J - \lambda_*\beta a_J\|}\right\}(y_J - \lambda_*\beta a_J),
\end{equation}
for each $J \in {\cal J}$.

If $\lambda_* = 0$, we can obtain the optimal solution $x^*$ directly from \eqref{solution1}. On the other hand, if $\lambda_* > 0$, plugging \eqref{solution1} into \eqref{subKKT2} above, we see that $\lambda_*$ can be obtained as a root of the following function:
\begin{equation}\label{G}
T(\lambda) := r- \sum_{J\in\mathcal{J}} \min\left\{\max\left\{1 - \frac{\beta}{\|y_J - \lambda\beta a_J\|}, 0\right\},\frac{M}{\|y_J - \lambda\beta a_J\|}\right\}a_J^T(y_J - \lambda\beta a_J).
\end{equation}
The solution $x^*$ can then be recovered via \eqref{solution1}. In our implementation, we apply Newton's method to find a root of $T$. Specifically, we follow the framework of Newton's method described in \cite{SoSv98} (see the algorithmic description in \cite[Algorithm~2.1]{SoSv98}).\footnote{Follow the notation in \cite[Algorithm~2.1]{SoSv98}, we set $\beta = \frac{1}{2}$, $\lambda = 10^{-4}$, and choose $G_k$ from the generalized Jacobian of $T$ at $\lambda_k$, $\mu_k = 10^{-4}\|T(\lambda_k)\|^{\frac{1}{2}}$ and $\rho_k = 0$ at each iteration $k$.} At iteration $k = 0$, we initialize the Newton's method at $\lambda = 0$, and at each subsequent iteration, we warm-start this algorithm using the $\lambda$ obtained in the previous iteration; we terminate the Newton's method when the approximate solution $\lambda_*$ satisfies $\|T(\lambda_*)\| \le 10^{-10}$. Moreover, as a safeguard, we also terminate when stepsizes are small: Specifically, we terminate the algorithm when the stepsize $\alpha_k$ in \cite[Algorithm~2.1]{SoSv98} falls below $10^{-10}$.


\subsection{Solving the subproblems in ESQM$_{\rm ls}$ for \eqref{E3} and \eqref{E4}}
The subproblem \eqref{subproblemhaha} (with $m = 1$) takes the following form:
\begin{equation}\label{solvsub1}
  \begin{array}{rl}
\min\limits_{(x,s)\in\R^{n+1}} & P_1(x) + \frac{1}{\beta} s + \frac{1}{2\beta}\|x - y\|^2 + \delta_{C}(x)\\
{\rm s.t.} & \langle a, x\rangle - r \leq s,\\
& s \geq 0,
  \end{array}
\end{equation}
where $C$ is a compact convex set, $y$, $a\in\R^n$, $\beta >0$ and $r\in \R$. The Slater condition holds trivially for the above problem, i.e., there exists $\hat x \in C$ and $\hat s > 0$ so that the first inequality constraint also holds strictly.

By \cite[Corollary~28.2.1, Theorem~28.3]{Ro70}, $(x^*, s_*)$ is optimal for \eqref{solvsub1} if and only if $(x^*, s_*)$ is feasible for \eqref{solvsub1} and there exists $\lambda_*\ge 0$ such that the following conditions hold:
\begin{enumerate}[{\rm (A)}]
  \item\label{subKKT11}  $0\in\partial (P_1 + \delta_{C})(x^*) + \frac{1}{\beta}(x^* - y) + \lambda_*a $ and $0\in \frac{1}{\beta} - \lambda_* + \mathcal{N}_{\R_+}(s_*)$;
  \item\label{subKKT22}  $\lambda_*(\langle a, x^*\rangle - r - s_*)=0$.
\end{enumerate}
Note that the first inclusion in \eqref{subKKT11} is equivalent to \eqref{solution}. Hence, when solving \eqref{E3}, one can again rewrite $x^*$ as \eqref{solution1}. Next, from the second inclusion of \eqref{subKKT11}, we have that $\lambda_* - \frac{1}{\beta} \leq 0$ and $s_*(\lambda_* - \frac{1}{\beta}) = 0$. Thus:
\begin{itemize}
  \item If $\lambda_* = 0$ or $\lambda_* = \frac{1}{\beta}$, {\color{black} we }can obtain the optimal solution $x^*$ from \eqref{solution1};
  \item If $0< \lambda_* < \frac{1}{\beta}$, then $s_* = 0$. In this case, \eqref{subKKT22} become $\lambda_*(\langle a, x^*\rangle - r)=0$, which is the same as \eqref{subKKT2} in Section~\ref{sec:append1}. The $\lambda_*$ is then obtained as a root of $T$ in \eqref{G}, and we find it using the Newton's method with the same parameter settings used in Section~\ref{sec:append1}. The $x^*$ is then recovered via \eqref{solution1}.
\end{itemize}

{\color{black}
\textbf{Funding} Yongle Zhang was supported partly by the National Natural Science Foundation of China (11901414) and (11871359) and Sichuan Science and Technology Program (2018JY0201). Guoyin Li was partially supported by a Future fellowship from Australian Research Council (FT130100038) and a discovery project from Australian Research Council (DP190100555). Ting Kei Pong was supported partly by Hong Kong Research Grants Council PolyU153000/20p.

\subsection*{Declarations}
\textbf{Conflict of interest} The second author is a guest editor of the topical collection "Mathematics of Computation and Optimisation" of this journal. }

\end{document}